\documentclass{amsart}

\usepackage{amsmath} 
\usepackage{amssymb}
\usepackage{mathrsfs}

\newtheorem{theorem}{Theorem}[section] 
\newtheorem{claim}[theorem]{Claim}

\newtheorem{lemma}[theorem]{Lemma}

\newtheorem{corollary}[theorem]{Corollary} 

\theoremstyle{definition}
\newtheorem{definition}[theorem]{Definition}
\newtheorem{example}[theorem]{Example}

\newtheorem{convention}[theorem]{Convention}

\theoremstyle{remark}

\def\forkindep{\mathrel{\raise0.2ex\hbox{\ooalign{\hidewidth$\vert$\hidewidth\cr\raise-0.9ex\hbox{$\smile$}}}}}

\newcommand{\Ax}{{\rm Ax}}
\newcommand{\cof}{{\rm cof}}
\newcommand{\rg}{{\rm rg}}

\newcommand{\supp}{{\rm supp}}

\newcommand{\cf}{{\rm cf}}

\newcommand{\dom}{{\rm dom}}

\newcommand{\GCH}{{\rm GCH}}

\newcommand{\bff}{{\mathbf f}}

\newcommand{\SP}{{\rm SP}}

\newcommand{\rest}{{\restriction}}

\newcommand{\mn}{{\medskip\noindent}}
\newcommand{\sn}{{\smallskip\noindent}}

\newcommand{\gC}{{\mathfrak C}}

\newcount\skewfactor
\def\mathunderaccent#1#2 {\let\theaccent#1\skewfactor#2
\mathpalette\putaccentunder}
\def\putaccentunder#1#2{\oalign{$#1#2$\crcr\hidewidth
\vbox to.2ex{\hbox{$#1\skew\skewfactor\theaccent{}$}\vss}\hidewidth}}

\newbox\noforkbox \newdimen\forklinewidth
\setbox0\hbox{$\textstyle\bigcup$}
\setbox1\hbox to \wd0{\hfil\vrule width \forklinewidth depth \dp0
                        height \ht0 \hfil}
\setbox\noforkbox\hbox{\box1\box0\relax}
\def\unionstick{\mathop{\copy\noforkbox}\limits}
\def\nonfork#1#2_#3{#1\unionstick_{\textstyle #3}#2}
\def\nonforkin#1#2_#3^#4{#1\unionstick_{\textstyle #3}^{\textstyle
    #4}#2}
\setbox0\hbox{$\textstyle\bigcup$}
\setbox1\hbox to \wd0{\hfil{\sl /\/}\hfil}
\setbox2\hbox to \wd0{\hfil\vrule height \ht0 depth \dp0 width
                                \forklinewidth\hfil}
\newbox\doesforkbox
\setbox\doesforkbox\hbox{\box1\box0\relax}
\def\nunionstick{\mathop{\copy\doesforkbox}\limits}

\def\fork#1#2_#3{#1\nunionstick_{\textstyle #3}#2}
\def\forkin#1#2_#3^#4{#1\nunionstick_{\textstyle #3}^{\textstyle
    #4}#2}

\newcommand{\stickT}{%
\setbox255=\hbox{\raise1ex\hbox{$\hspace{0.2pt}\,\bullet\,$}}
\mathord{\rlap{\hbox to\wd255{\hss\hbox{$|$}\hss}}
\box255}
}
\newcommand{\stickS}{%
\setbox255=\hbox{\raise0.6ex\hbox{$\scriptstyle\bullet$}}
\mathord{\rlap{\hbox to\wd255{\hss\hbox{$\scriptstyle|$}\hss}}
\box255}
}

\usepackage{hyperref}
\begin{document}

\title {$\le_{\SP}$ Can Have Infinitely Many Classes \\
1065}
\author {Saharon Shelah and Danielle Ulrich}
\address{Einstein Institute of Mathematics\\
Edmond J. Safra Campus, Givat Ram\\
The Hebrew University of Jerusalem\\
Jerusalem, 9190401, Israel\\
 and \\
 Department of Mathematics\\
 Hill Center - Busch Campus \\ 
 Rutgers, The State University of New Jersey \\
 110 Frelinghuysen Road \\
 Piscataway, NJ 08854-8019 USA}
\email{shelah@math.huji.ac.il}
\urladdr{http://shelah.logic.at}
\thanks{This material is based in part upon work supported by the National
Science Foundation under Grant No. 136974 and by European Research 
Council grant 338821. Paper 1065 on Shelah's list. The first author would like to thank for partially supporting this research an NSF-BSF 2021 grant with M. Malliairis, NSF 205182, BSF 3013005232. The second author partially supported
by Laskowski's NSF grants DMS-1308546 and DMS-2154101.
The authors thank Alice Leonhardt for the beautiful typing.
 References like \cite[Th0.2=Ly5]{Sh:950} means the label of Th.0.2
 is y5.  The reader should note that the version on the first author's website is
usually more updated than the one in the mathematical archive.}

% converted to my LaTeX Sept 2019



\date {September 24, 2019}

\begin{abstract}
Building off of recent results on Keisler's order, we show that
consistently, $\le_{\SP}$ has infinitely many classes. In particular,
we define the property of $\le k$-type amalgamation for simple
theories, for each $2 \le k < \omega$. If we let $T_{n, k}$ be the
theory of the random $k$-ary, $n$-clique free random hyper-graph, then
$T_{n,k}$ has $\le k-1$-type amalgamation but not $\le k$-type
amalgamation. We show that consistently, if $T$ has $\leq k$-type
amalgamation then $T_{k+1, k} \nleq_{\SP} T$, thus producing
infinitely many $\le_{\SP}$-classes. The same construction gives a
simplified proof of the theorem from \cite{ShelahSimple} that
consistently, the maximal $\le_{\SP}$-class is exactly the class of
non-simple theories. Finally, we show that consistently, if $T$ has 
$< \aleph_0$-type amalgamation, then $T \le_{\SP} T_{\rg}$, the theory 
of the random graph.
\end{abstract}

\maketitle
\numberwithin{equation}{section}
\setcounter{section}{-1}
\newpage

\section {Introduction}

\begin{convention}  
\label{a2}
\label{a1}
$T$ is always a complete theory in a countable language. We will fix a 
monster model $\gC \models T$ and work within it so $\gC = \gC_T$ but
if $T$ is clear from the context we do not mention it.

The first author introduced the following definition in
\cite{ShelahSimple}, although he had previously investigated the 
phenomenon in \cite{ShelahIso} (without giving it a name):
\end{convention}

\begin{definition}
\label{a4}
Suppose $\lambda \ge \theta$. Define $\SP_T(\lambda, \theta)$ to mean: 
for every $M \models T$ of size $\lambda$, there is a
$\theta$-saturated $N \models T$ of size $\lambda$ extending $M$. 
\end{definition}

\noindent
In this paper, we will restrict to the following special case:
\begin{definition}
\label{a6}
1) Say that $(\theta,\lambda)$ is a nice pair if $\theta$ is a regular
cardinal and $\lambda \ge \theta$ has $\lambda = 
\lambda^{\aleph_0}$.

\noindent
2)  Given $T_0, T_1$ complete first order theories, say that $T_0
\le_{\SP} T_1$ if whenever $(\theta, \lambda)$ is a nice pair, if 
$\SP_{T_1}(\lambda,\theta)$ then $\SP_{T_0}(\lambda,\theta)$.

Thus, $\le_{\SP}$ is a pre-ordering of theories which measures how
difficult it is to build saturated models. 

In \cite{ShelahIso}, the first author proves: the stable theories are
the minimal $\SP$-class, and non-simple theories are always maximal. 
In \cite{ShelahSimple}, the first author additionally proves that 
consistently, non-simple theories are exactly the maximal class. 

Recently, there has been substantial progress on Keisler's order 
$\trianglelefteq$, another pre-ordering of theories which measures 
how difficult it is to build saturated models; see for instance  
\cite{Optimals} and \cite{InfManyClass} by the first author and 
Malliaris. In particular, in \cite{InfManyClass} it is shown that 
Keisler's order has infinitely many classes, these being seperated 
by certain amalgamation properties. In this paper we use similar ideas 
to continue investigation of $\le_{\SP}$.

In \S2 we summarize what is already known on $\le_{\SP}$.

In \S3, we introduce several amalgamation-related properties of
forcing notions (Definition ~\ref{DefOfAPProps}), and show that it is 
preserved under iterations in a suitable sense 
(Theorem~\ref{IterationPreserves}). In light of this, we define a
class of forcing axioms (Definition~\ref{DefOfAxioms}); these are 
closely related to the forcing axiom $\Ax_{\mu_0}$, defined by the first 
author in \cite{ShelahForcing} and used to demonstrate the consistent 
maximality of non-simple theories under $\le_{\SP}$ in
\cite{ShelahSimple}. 
However, the forcing axioms we develop are designed specifically for
what we want and have been simplified somewhat.

In \S4, we define and prove some helpful facts about non-forking diagrams of models.

In \S5, we introduce, for each $3 \le k < \omega$, a property of
simple theories called $< k$-type amalgamation
(Definition~\ref{DefOfTypeAP}), and discuss some of its
properties. For example, if for $n > k$ we let $T_{n,k}$ be the theory
of the $k$-ary, $n$-clique free hypergraph, then if $k \ge 3$,
$T_{n,k}$ has $<k$-type amalgamation but not $< k+1$-type
amalgamation.  We also show that if $T$ has $<\aleph_0$-type
amalgamation (i.e., $< k$-type amalgamation for all $k$), then
$\SP_T(\lambda,\theta)$ holds whenever we have that there is some
$\theta \le \mu \le \lambda$ with $\mu^{<\theta} \le \lambda$ and
$2^\mu \ge \lambda$ (Theorem~\ref{SatPortion}). This implies that 
if the singular cardinals hypothesis holds, then whenever $T$ has 
$<\aleph_0$-type amalgamation, then $T \le_{\SP} T_{\rg}$, where
$T_{\rg}$ is the theory of the random graph.

In \S6, we put everything together to show that consistently, for all
$k \ge 3$, if $T$ has the $< k$-type amalgamation property, then
$T_{k,k-1} \nleq_{\SP} T$ (Theorem~\ref{InfManyClasses}). In
particular, for $k < k'$, $T_{k+1, k} \nleq_{\SP} T_{k'+1,k'}$; this
is similar to the situation for Kiesler's order in \cite{InfManyClass}.

By a forcing notion, we mean a pre-ordered set $(P,\le^P)$ such that
$P$ has a least element $0^P$ (pre-order means that $\le^P$ is
transitive); we are using the convention where $p \le q$ means $q$ is
a stronger condition than $p$. That is, when we force by $P$ we add a
generic ideal, rather than a generic filter. Thus, a finite sequence 
$(p_i: i < k)$ from $P$ is compatible if it has an upper bound in $P$. 
\end{definition}
\newpage

\section {Background} \label{1}
\bigskip

The following theorem is closely related to the classical 
Hewitt-Marczewski-Pondiczery theorem of topology; the special case $\theta = \aleph_0$ is implied by Theorem 8 of \cite{Partitions}, and the general case is also noted there. It will be central for our investigations.
\bigskip

\begin{theorem}
\label{ER}
Suppose $\theta \le \mu \le \lambda$ are infinite cardinals such that
$\theta$ is regular, $\mu = \mu^{<\theta}$, and $\lambda \le 2^\mu$. 
Then there is a sequence $(\bff_\gamma:\gamma < \mu)$ from ${}^\lambda
\mu$ such that for all partial functions $f$ from $\lambda$ to $\mu$
of cardinality less than $\theta$, there is some $\gamma < \mu$ such
that $\bff_\gamma$ extends $f$. 
If $\lambda > 2^\mu$ then this fails, in fact, there is no sequence $(\mathbf{f}_\gamma: \gamma < \mu)$ from $\,^\lambda 2$ such that for all partial functions $f$ from $\lambda$ to $2$ of cardinality less than $\theta$, there is some $\gamma < \mu$ such that $\mathbf{f}_\gamma$ extends $f$.
\end{theorem}

We will also want the following technical device, which will allow us
to apply Theorem~\ref{ER} to conclude $\SP_T(\lambda,\theta)$ holds. 
Here is the idea: suppose $M \models T$ with $|M| \le \lambda$, and we
want to find some $\theta$-saturated $N \succeq M$ with $|N| \le
\lambda$.  To do this, we will always first find some $N_0 \succeq M$ 
with $|N_0| \le \lambda$ which realizes every type over $M$ of
cardinality less than $\theta$, and then we iterate $\theta$-many
times.  The key step is to find $N_0$, and the following definitions 
capture when this is possible.

\begin{definition}
% \label{}
1) Suppose $T$ is a simple theory, $\theta$ is a regular uncountable
cardinal, and $M_* \preceq M \models T$. then let
$\Gamma^{\theta}_{M,M_*}$ be the forcing notion of all partial types
$p(x)$ over $M$ of cardinality less than $\theta$ which do not 
fork over $M_*$, ordered by inclusion, where $x$ is a single variable. Also, if $p_*(x)$ is a 
complete type over $M_*$, then let $\Gamma^\theta_{M, p_*} \subseteq 
\Gamma^{\theta}_{M,M_*}$ be the set of all $p(x)$ which extend $p_*(x)$.

\noindent
2) Given $(\theta,\lambda)$ a nice pair and given $\mu$ with $\theta
\le \mu \le \lambda$, define $\SP^1_{T}(\lambda,\mu,\theta)$ to mean:
for every $M \models T$ of size $\leq \lambda$ and for every countable 
$M_* \preceq M$, there are complete types $p_\gamma(x):\gamma < \mu$
over $M$ which do not fork over $M_*$, such that whenever $p(x) \in 
\Gamma^\theta_{M,M_*}$, then $p(x) \subseteq p_\gamma(x)$ for some 
$\gamma < \mu$.

\noindent
3) Given in addition a fixed countable $M_* \models T$ and type $p_*(x)$ over
$M_*$, define $\SP^1_{T,p_*}(\lambda,\mu,\theta)$ similarly: whenever 
$M \succeq M_*$ has size at most $\lambda$, there are  are complete, 
non-forking extensions $p_\gamma(x): \gamma < \mu$ of $p_*(x)$ to $M$,
such that whenever $p(x) \in \Gamma^\theta_{M,p_*}$, then $p(x)
\subseteq p_\gamma(x)$ for some $\gamma < \mu$.
\end{definition}

Note that if $\mu \ge 2^{\aleph_0}$, then
$\SP^1_T(\lambda,\mu,\theta)$ if and only if
$\SP^1_{T,p_*}(\lambda,\mu,\theta)$ for every complete type $p_*(x)$
over a countable model $M_*$ (the forward direction is unconditional 
in $\mu$, but for the reverse direction, we need to concatenate
witnesses for each $p_*(x)$, of which there are $2^{\aleph_0}$-many). 
In particular this holds when $\mu = \lambda$, since $\lambda^{\aleph_0} = \lambda$.

\noindent
The following is an important example. Let $T_{\rg}$ be the theory of the random graph, i.e. the model completion of the theory of graphs. $T_{rg}$ admits quantifiers, and given $A \subseteq B$ and $p(x) \in S(B)$, $p$ forks over $A$ if and only if $p$ is realized in $B \backslash A$.
\begin{example}
\label{TrgComp}
\label{2.3}
Suppose $(\theta,\lambda)$ is a nice pair and suppose $\mu$ is a 
cardinal with $\mu = \mu^{<\theta}$ and $\theta \le \mu \le \lambda$. 
Then $\SP^1_{T_{\rg}}(\lambda,\mu,\theta)$ holds if and only if
$\lambda \le 2^\mu$; and this is equivalent to $\SP^1_{T_{\rg},p_*}
(\lambda,\mu,\theta)$ holding for some or any nonalgebraic complete 
type $p_*(x)$ over a countable model $M_*$.
\end{example}

\begin{proof}
Suppose $M \models T$ has size $\le \lambda$.

Then the non-algebraic types in $S^1(A)$
correspond naturally to functions from $A$ to $2$, and so this
is just a restatement of Theorem~\ref{ER}.
\end{proof}

\begin{theorem}\label{SPVariants}
\label{2.4}
Suppose $T$ is a simple theory (in a countable language, as always see
\ref{a1}). 

Suppose $(\theta,\lambda)$ is a nice pair:
\mn
\begin{enumerate}
\item[(A)]  If $\SP^1_T(\lambda,\lambda, \theta)$, then $\SP_T(\lambda, \theta)$.
\sn
\item[(B)] If $T = T_{rg}$ and $\SP_T(\lambda, \theta)$ then $\SP_T(\lambda, \lambda, \theta)$.
\item[(C)]  Suppose $p_*(x)$ is a complete type over a countable model
$M_* \models T$, and $\SP^1_{T,p_*}(\lambda,\lambda,\theta)$ holds, and
$\cof(\lambda) < \theta$. Then for some $\mu$ with $\theta \le \mu <
\lambda,\SP^1_{T,p_*}(\lambda,\mu,\theta)$ holds.
\item[(D)] Suppose $2^{\aleph_0} < \cof(\lambda) < \theta$. Suppose $\SP^1_T(\lambda, \lambda, \theta)$ holds. Then $\SP^1_T(\lambda, \mu, \theta)$ holds for some $\mu < \lambda$.
\end{enumerate}
\end{theorem}

\begin{proof}
% (A) \quad forward direction: Suppose $M \models T$ has size $\le
% \lambda$, and $M_* \preceq M$ is countable. Choose $N \succeq M$, a 
% $\theta$-saturated model of size $\lambda$. Let $(a_\alpha: \alpha < \lambda)$ enumerate all $a \in N$ such that $\mbox{tp}(a/M)$ does not fork over $M_*$. Clearly  

% As $\lambda > \theta \ge \beth_\omega$ by \cite{Sh:460}, there is
% $\kappa = \cf(\kappa) < \beth_\omega$ such that $\lambda =
% \lambda^{[\kappa]}$ which means that there is $\cS \subseteq
% [\lambda]^\kappa$ of cardinality $\lambda$ such that for every $v \in
% [\lambda]^\kappa$ there is $u \in [u]^\kappa$ such that $u \in \cS$.
% For $u \in \cP$ let $p_u(x)$ be $\{\varphi(x,\bar b):\varphi =
% \varphi(x,\bar y) \in \bbL(T),\bar b \in {}^{\ell g(\bar b)} N$ and
% for some $w \in [u]^{< \kappa}$ we have $\alpha \in u \backslash h
% \Rightarrow N \models \varphi[a_\alpha,\bar b]$.

% Clearly $p_u(x)$ is a type in $N$.  Let $\cS_1 = \{u \in \cS:u \in S$
% and $p_u(x)$ extends $p_*(x)$ and does not fork over $M_*\}$ and for
% each $u \in \cS_1$ let $q_u(x) \in \bfS(M)$ extends $p_u(x)$ and does
% not fork over $M_*\}$.  Easily $\cP = \{q_u(x):u \in \cS_1\}$ is as required.

\noindent
(A) \quad Suppose $M \models T$ has size $\le
\lambda$. 
Using $\SP^1_T(\lambda,\lambda,\theta)$, we can find $N \succeq M$ of 
size $\lambda$, such that every partial type $p(x)$ over $M$ of
cardinality less than $\theta$ is realized in $N$, using every type in
$M$ does not fork over some countable submodel of $M$ (we are also using 
$\lambda = \lambda^{\aleph_0}$, so there are only $\lambda$-many
countable elementary submodels $M_*$ of $M$). If we iterate this 
$\theta$-many times then we will get a $\theta$-saturated model of $T$. 

\noindent (B): Suppose $M \models T_{rg}$ has size $\leq \lambda$ and $M_* \preceq M$ is countable. Choose $N \succeq M$, a $\theta$-saturated model of size $\lambda$. Let $a_\alpha: \alpha < \lambda$ enumerate $N$. For each $\alpha < \lambda$ let $p_\alpha(x)$ be the type over $M$ asserting $x \not= a$ for each $a \in M$, and $R(x, a)\in p_\alpha$ if and only if $R(a_\alpha, a)$ holds for each $a \in M$. This is a complete type over $M$ which does not fork over $\emptyset$. Then $\{p_\alpha(x): \alpha < \lambda\}$ along with all algebraic types over $M_*$ witness $\mbox{SP}^1_{T_{rg}}(\lambda, \lambda, \theta)$.

\noindent
(C): Suppose towards a contradiction that
$\SP^1_{T,p_*}(\lambda,\mu,\theta)$ failed for all $\theta \le \mu <
\lambda$. Write $\kappa = \cof(\lambda)$, and let $(\mu_\beta:\beta <
\kappa)$ be a cofinal sequence of cardinals in $\lambda$ with each $u_\beta \geq \theta$. For each 
$\beta < \kappa$, choose $M_\beta \succeq M_*$ with $|M_\beta| \le 
\lambda$, witnessing that $\SP^1_{T,p_*}(\lambda,\mu_\beta,\theta)$
fails.  We can suppose that $(M_\beta:\beta < \kappa)$ is independent over $M_*$.

Let $N \models T$ have size $\le \lambda$ such that each $M_{\beta}
\preceq N$. Then by $\SP^1_{T,p_*}(\lambda,\lambda,\theta)$, we can 
find $(q_\alpha(x):\alpha < \lambda)$ such that each
$q_\alpha(x)$ extends $p_*(x)$, does not fork over $M_*$ and
whenever $q(x) \in \Gamma^\theta_{N,p_*}$, then $q(x) \subseteq q_\alpha(x)$ for 
some $\alpha < \lambda$.

For each $\beta < \kappa$, we can by hypothesis choose $p_\beta(x) \in 
\Gamma^{\theta}_{M^\beta,p_*}$ such that $p_\beta(x) \nsubseteq
q_\alpha(x)$ for any $\alpha < \mu_\beta$; note that still $p(x)
\supseteq p_*(x)$.  By the independence 
theorem for simple theories, $p(x) := \bigcup\limits_{\beta < \kappa} 
p_\beta(x)$ does not fork over $M_*$. Hence $p(x) \subseteq
q_\alpha(x)$ for some $\alpha < \lambda$. Choose $\beta < \kappa$ with 
$\alpha < \mu_\beta$; then this implies that $p_\beta(x) \subseteq
q_\alpha(x)$, a contradiction.

(D): Enumerate, up to isomorphism, all types over countable models $(p_\alpha(x, M_\alpha): \alpha < 2^{\aleph_0})$. For each $\alpha < 2^{\aleph_0}$, $\SP^1_{T, p_\alpha}(\lambda, \lambda, \theta)$ holds, so by (C) there is $\mu_\alpha < \lambda$ such that $\SP^1_{T, p_\alpha}(\lambda, \mu_\alpha, \theta)$ holds. Let $\mu$ be the supremum of $2^{\aleph_0}$ and $\{\mu_\alpha: \alpha < 2^{\aleph_0}\}$; then $\mu < \lambda$ and easily $\SP^1_T(\lambda, \mu, \theta)$ holds. 
\end{proof}

\noindent
Finally, the following theorem is a collection of most of what has
been previously known on $\le_{\SP}$.
\begin{theorem}
\label{Old}
\label{2.5}  
Suppose $T$ is a complete first order theory in a countable language. 

Suppose $(\theta, \lambda)$ is a nice pair:
\mn
\begin{enumerate}
\item[(A)]  If $\lambda = \lambda^{<\theta}$, then
  $\SP_T(\lambda,\theta)$ holds; if $T$ is non-simple then the converse
  is true as well. Thus non-simple theories are all
  $\le_{\SP}$-maximal. 
(See \cite{ShelahSimple} Conclusion 4.6 and Theorem 4.7.)
\sn
\item[(B)]  $T_{\rg}$ is the $\le_{\SP}$-minimal unstable theory. 
(Theorem 4.8 of \cite{ShelahSimple}.)
\sn
\item[(C)]  If $T$ is stable, then $\SP_T(\lambda,\theta)$ holds 
(see \cite{ShelahSimple} Theorem 4.7(2)). 
\sn
\item[(D)]  If $\lambda$ is a strong limit with $\cof(\lambda) <
  \theta$ (and as $\lambda = \lambda^{\aleph_0}$, we have $\aleph_0 < \cof(\lambda))$, and if
  $\SP_T(\lambda,\theta)$ holds, then $T$ is stable. 
(Theorem 4.7(4) of \cite{ShelahSimple}.) Thus the stable theories
are exactly the minimal $\le_{\SP}$-class.  Also, under $\GCH$, all 
unstable theories are maximal.
\sn
\item[(E)]  If $\theta \le \mu \le \lambda$ and $\mu^{<\theta} = \mu$
  and $\lambda \le 2^\mu$, then $\SP_{T_{\rg}}(\lambda,\theta)$
holds. (This is Exercise VIII 4.5 in \cite{ShelahIso}.)
\sn
\item[(F)]  It is consistent that there exists a nice pair
  $(\theta,\lambda)$ such that for all simple
  $T,\SP_T(\theta,\lambda)$ holds. 
Hence, it is consistent that the non-simple theories are exactly 
the $\le_{\SP}$-maximal class. (This is Theorem 4.10 of \cite{ShelahSimple}.)
\end{enumerate}
\end{theorem}

\noindent
For the reader's convenience, we prove (A) through (E), making use 
of the language of $SP^1$. Theorem (F) will be a special case of our 
main theorem, namely Theorem~\ref{InfManyClasses}(B).
\begin{proof}
(A): By standard arguments, if $\lambda^{<\theta} = \lambda$ then 
$\SP_T(\lambda,\theta)$ holds. Suppose $T$ is non-simple, and
$\SP_T(\lambda,\theta)$ holds, and suppose towards a contradiction 
that $\lambda^{<\theta} > \lambda$. Choose a formula $\phi(x,y)$ with 
the tree property (possibly $y$ is a tuple).

Let $\kappa < \theta$ be least such that $\lambda^\kappa > \lambda$. 
Choose $M \models T$ and $(a_\eta:\eta \in {}^{<\kappa} \lambda)$ such 
that for all $\eta \in {}^\kappa \lambda,p_\eta(x) := \{\phi(x,a_{\eta
  \rest_\beta}): \beta < \kappa\}$ is consistent, and for all $\eta
\in {}^{<\kappa} \lambda$ and for all $\alpha < \beta < \lambda,
\phi(x,a_{\eta \frown (\alpha)})$ and $\phi(x,a_{\eta \frown
  (\beta)})$ are inconsistent. Note that each $|p_\eta(x)| < \theta$; 
but clearly if $N \succeq M$ realizes each $p_\eta(x)$ then $|N| \ge 
\lambda^\kappa > \lambda$.

\noindent
(B): Suppose $T$ is unstable; we show $T_{\rg} \le_{\SP} T$. 
By (A), this is true if $T$ is non-simple, so we can suppose that $T$ is
simple, hence has the independence property via some formula
$\phi(x,y)$.  Now suppose $(\theta,\lambda)$ is a nice pair. 
By Theorem~\ref{SPVariants}(B), it suffices to show that if 
$\SP_T(\lambda,\theta)$ holds, then
$\SP^1_{T_{\rg}}(\lambda,\lambda,\theta)$ holds.  
Choose some $(a_\alpha:\alpha <
\lambda)$ from $\gC$ such that for all $\bff:\lambda \rightarrow 2,
\{\phi(x,a_\alpha)^{\bff(\alpha)}:\alpha < \lambda\}$ is consistent. 
By $\SP_T(\lambda,\theta)$ we can find some $\theta$-saturated 
$M \prec \gC$ with $|M| \le \lambda$ and each $a_\alpha \in M$.

Suppose $N \models T_{\rg}$ has cardinality $\lambda$ , say 
$N = \{b_\alpha:\alpha < \lambda\}$ without repetitions. For each $c
\in M,$ let $p_c(x)$ be the complete nonalgebraic type over $N$, defined 
by putting $R(x, b_\alpha) \in p_c(x)$ if and only if $M \models 
\phi(c,a_\alpha)$. Then recalling the proof of \ref{2.3} this witnesses
$\SP^1_{T_{\rg}}(\lambda,\lambda,\theta)$ holds (since $|M| \le \lambda$).

\noindent
(C): Suppose $T$ is stable. It suffices to show that
$\SP^1_T(\lambda,\lambda,\theta)$ holds. But this is clear: 
given $M \models T$ of size $\le \lambda$ and $M_* \preceq M$
countable, there are $\le 2^{\aleph_0} \le \lambda^{\aleph_0} = \lambda$
many types over $M$ that do not fork over $M_*$, seeing as types 
over $M_*$ are stationary. 

\noindent
(D): Suppose towards a contradiction that $\SP_{T}(\lambda,\theta)$
holds for some unstable $T$. Then in particular
$\SP_{T_{\rg}}(\lambda,\theta)$ holds.  Let $p_*(x)$ be a complete
non-algebraic type over some countable $M_* \models T_{\rg}$. 
By Theorem~\ref{SPVariants} we can find $\theta \le \mu < \lambda$
such that $\SP^1_{T_{\rg},p_*}(\lambda,\mu,\theta)$ holds. By possibly 
replacing $\mu$ with $\mu^{<\theta}$ we can suppose $\mu =
\mu^{<\theta}$. Then this contradicts Example~\ref{TrgComp}, since
$2^\mu < \lambda$.

\noindent
(E): By Example~\ref{TrgComp} and Theorem~\ref{SPVariants}(A).

\noindent
(F): See \cite{ShelahSimple} or \cite{GIL02}.
\end{proof}

\noindent
If the singular cardinals hypothesis holds, then we can say more. 
Recall that
\begin{definition}
% \label{}
The singular cardinals hypothesis states that if $\lambda$ 
is singular and $2^{\cof(\lambda)} < \lambda$, then
$\lambda^{\cof(\lambda)} = \lambda^+$. (Note that $2^{\cof(\lambda)}
\ne \lambda$ since $\cof(2^\kappa) > \kappa$ for all cardinals
$\kappa$, by K\"{o}nig's theorem.)

The failure of the singular cardinals hypothesis is a large cardinal
axiom; see Chapter 5 of \cite{Jech}.
\end{definition}

\noindent
We want the following simple lemma.

\begin{lemma}
\label{SCHLemma}
Suppose the singular cardinals hypothesis holds. Suppose $\theta$ is
regular, $\lambda \ge \theta$, $\lambda^{<\theta} > \lambda$, and
$2^{<\theta} \le \lambda$. Then for every $\mu < \lambda,
\mu^{<\theta} < \lambda$. Further, $\lambda$ is singular of cofinality $<\theta$.
\end{lemma}

\begin{proof}
First of all, note that $2^{<\theta} < \lambda$, as otherwise 
$\lambda^{<\theta} = \lambda$.

Now suppose towards a contradiction there were some $\mu < \lambda$
with $\mu^{<\theta} \ge \lambda$; then necessarily $\mu^{<\theta} >
\lambda$, as otherwise again $\lambda^{<\theta} = \lambda$. We can
choose $\mu$ least with $\mu^{<\theta} > \lambda$. 
Let $\kappa < \theta$ be least such that $\mu^\kappa > \lambda$.

Note that $2^\kappa < \mu$, as otherwise $2^\kappa = (2^\kappa)^\kappa
\ge \mu^\kappa > \lambda$, contradicting $2^{<\theta} < \lambda$. 
Thus, by a consequence of the singular cardinals hypothesis 
(Theorem~{5.22}(ii)(b),(c) of \cite{Jech}), $\mu^\kappa \le
\mu^+$. But since $\mu < \lambda$, $\mu^+ \le \lambda$, so this is a contradiction.

To finish, suppose towards a contradiction that $\cof(\lambda) \ge
\theta$.  Then $\lambda^{<\theta} = \lambda + \sup\{\mu^{<\theta}:\mu <
\lambda\} = \lambda$, a contradiction.
\end{proof}

\begin{theorem}
\label{SCHTheorem}
Suppose the singular cardinals hypothesis holds, and suppose
$(\theta,\lambda)$ is a nice pair. Then $\SP^1_T(\lambda,\lambda,\theta)$ holds
if and only if $T$ is stable, or $\lambda = \lambda^{<\theta}$,
or else $T$ is simple and for every complete type $p_*(x)$ over a
countable model $M_* \models T$, there is some $\mu$ with $\theta \le \mu <
\lambda$ and with $\mu^{<\theta} = \mu$ and $2^\mu \ge \lambda$, 
such that $\SP^1_{T, p_*}(\lambda, \mu, \theta)$ holds.
\end{theorem}

\begin{proof}
If $T$ is stable or $\lambda = \lambda^{<\theta}$, then 
$\SP^1_T(\lambda,\lambda,\theta)$ holds. 
If $T$ is non-simple and $\lambda < \lambda^{<\theta}$, then $\SP^1_T(\lambda, \lambda, \theta)$ fails by Theorem~\ref{Old}(A) and Theorem~\ref{SPVariants}(A).
Thus we can assume $T$ is unstable, simple (hence has the independence
property) and $\lambda < \lambda^{<\theta}$. 

It suffices to show that $\SP^1_T(\lambda,\lambda,\theta)$ holds if
and only if for every complete type $p_*(x)$ over a countable model
$M_*$, there is some $\theta \le \mu < \lambda$ with $\mu^{<\theta} =
\mu$ and $2^\mu \ge \lambda$, such that $\SP^1_{T,p_*}(\lambda,\mu,\theta)$ holds. 

Suppose first $\SP^1_T(\lambda,\lambda,\theta)$ holds, and $p_*(x)$ is
given. Since $T$ is unstable with the independence property,
$\SP^1_T(\lambda,\lambda,\theta)$ clearly implies that $2^{<\theta}
\le \lambda$. Hence, by Lemma~\ref{SCHLemma}, $\lambda$ is singular
with $\cof(\lambda) < \theta$, and there are cofinally many $\mu <
\lambda$ with $\mu^{<\theta} = \mu$. By Theorem~\ref{Old}(D),
$\lambda$ is not a strong limit. Thus by Theorem~\ref{SPVariants}(C),
we can find $\theta \le \mu < \lambda$ such that $\mu = \mu^{<\theta}$
and $2^\mu \ge \lambda$ and $\SP^1_{T,p_*}(\lambda,\mu,\theta)$ holds.

Conversely, we have in particular that each
$\SP^1_{T,p_*}(\lambda,\lambda,\theta)$ holds; since $\lambda =
\lambda^{\aleph_0} \ge 2^{\aleph_0}$ we get that 
$\SP^1_T(\lambda,\lambda,\theta)$ holds.
\end{proof}
\newpage

\section {Forcing Axioms} \label{3}
\bigskip

In this section, we introduce the forcing axioms which will produce
the desired behavior in $\SP$. It is well-known that the countable
chain condition is preserved under finite support iterations; we aim
to find generalizations to the $\kappa$-closed, $\kappa^+$-c.c. context.
\bigskip

\begin{definition}
% \label{}
For a regular cardinal $\theta$ and sets $X,Y$, define $P_{XY\theta}$ to the
forcing notion of all partial functions from $X$ to $Y$ of cardinality
less than $\theta$, ordered by inclusion. Note that $P_{XY \theta}$
has the $|Y^{<\theta}|^+$-c.c. by the $\Delta$-system lemma and is $\theta$-closed.
\end{definition}

\begin{definition}
\label{DefOfAPProps}
Suppose $P,Q$ are forcing notions, and suppose $k \ge 3$ is a cardinal
(typically finite). Then say that $P \rightarrow_{k} Q$ if there is a
dense subset $P_0$ of $P$ and a map $F:P_0 \rightarrow Q$ such that
for all sequences $(p_i: i < i_*)$ from $P_0$ with $i_* < k$, if
$(F(p_i):i < i_*)$ is compatible in $Q$ (that is, has a common upper
bound), then $(p_i:i < i_*)$ has a least upper bound in $P$; we write 
$F:(P,P_0) \rightarrow_k Q$. Say that $P \rightarrow_{k}^w Q$ 
(where $w$ stands for weak) if there is a map $F:P \rightarrow Q$ such
that whenever $(p_i: i < i_*)$ is a sequence from $P$ with $i_* < k$,
if $(F(p_i):i < i_*)$ is compatible in $Q$, then $(p_i:i < i_*)$ is
compatible in $P$.

Suppose $P$ is a forcing notion, $\aleph_0< \theta \le \mu$ are
cardinals with $\theta$ regular, and $3 \le k \le \theta$ is a 
cardinal (often finite). Then say that $P$ has the
$(<k,\mu,\theta)$-amalgamation property if every ascending chain from
$P$ of length less than $\theta$ has a least upper bound in $P$, and
for some set $X$, $P \rightarrow_{k} P_{X \mu \theta}$.
\end{definition}

\noindent
For example, $P_{X \mu \theta}$ has the $(<k, \mu, \theta)$-amalgamation property.
\newline
The following lemma sums up several obvious facts.
\begin{lemma}
\label{ObviousLemmasForcingAP}
Suppose $\aleph_0 < \theta \le \mu$ are cardinals with $\theta = \cf(\lambda) >
\aleph_0$, and $3 \le k \le \theta$ is a cardinal. 
\mn
\begin{enumerate}
\item   If $P \rightarrow_{k} Q$ and $Q \rightarrow_{k}^w Q'$ 
then $P \rightarrow_k Q'$.
\sn
\item   If $P,Q$ have the $(<k,\mu,\theta)$-amalgamation property,
  then $P$ forces that $\check{Q}$ has the
  $(<k,|\mu|,\theta)$-amalgamation property. (We write $|\mu|$ because
  possibly $P$ collapses $\mu$.) (This is where we use $k
  \le \theta$.)
\sn
\item  Suppose $P$ has the $(<k,\mu,\theta)$-amalgamation property for
  some $k \ge 3$. Then $P$ is $\theta$-closed (hence
  $<\theta$-distributive) and $(\mu^{<\theta})^+$-c.c.
\sn
\item  If $P$ is $\theta$-closed and has the least upper bound
  property, then $P$ has the $(<k,\mu,\theta)$-amalgamation property
  if and only if $P \rightarrow_k^w P_{\lambda \mu \theta}$ for some $\lambda$.
\end{enumerate}
\end{lemma}

\noindent
We note the following:
\begin{lemma}
\label{muNice}
Suppose $\aleph_0 < \theta \le \mu$ are cardinals with $\theta$
regular, and $3 \le k \le \theta$. Then $P$ has the
$(<k,\mu,\theta)$-amalgamation property if and only if $P$ has the
$(<k,\mu^{<\theta},\theta)$-amalgamation property. 
\end{lemma}

\begin{proof}
Define $\mu' = \mu^{<\theta}$, and let $\lambda$ be a cardinal. It
suffices to show there is a cardinal $\lambda'$ such that $P_{\lambda
  \mu' \theta} \rightarrow_k^w P_{\lambda' \mu \theta}$, by 
Lemma~\ref{ObviousLemmasForcingAP}(1). Write $Y' = {}^{<\theta} \mu$;
it suffices to find a set $X'$ such that $P_{\lambda Y' \theta} 
\rightarrow_k^w P_{X' \mu \theta}$.

Let $X' = \lambda \times (\theta+1)$. Define $F:P_{\lambda Y' \theta}
\rightarrow P_{X' \mu \theta}$ as follows. Let $f \in P_{\lambda Y'
  \theta}$ be given. Let $\dom(F(f)) = \{((\gamma,\delta):\gamma \in
\dom(f)$ and either $\delta < \dom(f(\gamma))$ or $\delta =
\theta\}$.  Define $F(f)(\gamma,\delta) = f(\gamma)(\delta)$ if
$\delta < \theta$, and otherwise $F(f)(\gamma,\theta)=
\dom(f(\gamma))$. 
Clearly this works. 
\end{proof}

\noindent
The following is key; it states that the
$(<k,\mu,\theta)$-amalgamation property is preserved under 
$<\theta$-support iterations. Note that it follows that the
$(<k,\mu,\theta)$-amalgamation property is preserved under $<\theta$-support products.
\begin{theorem}
\label{IterationPreserves}
Suppose $\theta$ is a regular uncountable cardinal, $\mu \ge \theta$
and $3 \le k \le \theta$. Suppose $(P_\alpha:\alpha \le
\alpha_*),(\dot{Q}_\alpha:\alpha < \alpha_*)$ is a $<\theta$-support
forcing iteration, such that each $P_\alpha$ forces that
$\dot{Q}_\alpha$ has the $(<k,|\mu|,\theta)$-amalgamation property. 
Then $P_{\alpha_*}$ has the $(<k,\mu,\theta)$-amalgamation property.
\end{theorem}

\begin{proof}
Let $\lambda$ be large enough.

Inductively, choose $(P^0_\alpha: \alpha \le \alpha_*,
\dot{Q}^0_\alpha:\alpha < \alpha_*)$ a $<\theta$-support forcing
iteration, and $(\dot{F}_\alpha: \alpha<\alpha_*)$, such that each
$P^0_\alpha$ is dense in $P_\alpha$ (and hence $<\theta$-distributive), and each $P_\alpha$ forces
$\dot{F}_\alpha:(\dot{Q}_\alpha,\dot{Q}^0_\alpha) \rightarrow_k 
\check{P}_{\lambda \mu \theta}$. There is a subtlety here: $\dot{Q}^0_\alpha$ needs to be a $P^0_\alpha$-name for $\dot{Q}_\alpha$, not just a $P_\alpha$-name. This follows from a general fact that if $P$ is a forcing notion and $P_0$ is dense in $P$ then any $P$-name is forced to be equivalent to a $P_0$-name; this can be checked by an induction on the foundation rank of $P$-names.

By revising the choice of $\dot{Q}^0_\alpha$ and $\dot{F}_\alpha$, we can suppose $\dot{Q}^0_\alpha$ contains the minimal element $0^{\dot{Q}_\alpha}$ of $\dot{Q}_\alpha$ and we can suppose $\dot{F}_\alpha$ is forced to take $0^{\dot{Q}_\alpha}$ to the empty function in $\check{P}_{\lambda \mu \theta}$.

\bigskip

\begin{claim}
% \label{}
For each $\gamma_*< \theta$, if $(p_\gamma: \gamma < \gamma_*)$ is an
ascending chain from $P_{\alpha_*}$, then it has a least upper bound
$p$ in $P_{\alpha_*}$, such that $\supp(p) \subseteq 
\bigcup\limits_{\gamma < \gamma_*} \supp(p_\gamma)$.
\end{claim}

\begin{proof}
By induction on $\alpha \le \alpha_*$, we construct $(q_\alpha:\alpha
\le \alpha_*)$ such that each $q_\alpha \in P_\alpha$ with
$\supp(q_\alpha) \subseteq \bigcup_{\gamma < \gamma_*} \supp(p_\gamma)
\cap \alpha$, and for $\alpha < \beta \le \alpha_*,q_{\beta}
\rest_\alpha = q_\alpha$, and for each $\alpha \le \alpha_*$,
$q_\alpha$ is a least upper bound of $(p_\gamma \restriction_\alpha: 
\gamma < \gamma_*)$ in $P_{\alpha}$. At limit stages there is nothing
to do; so suppose we have defined $q_\alpha$. If $\alpha \notin 
\bigcup\limits_{\gamma < \gamma_*} \supp(p_\gamma)$ then let
$q_{\alpha+1} = q_\alpha \frown (0^{\dot{Q}_\alpha})$. Otherwise,
since $q_\alpha$ forces that $(p_\gamma(\alpha): \gamma < \gamma_*)$
is an ascending chain from $\dot{Q}_\alpha$, we can find $\dot{q}$, a 
$P_\alpha$-name for an element of $\dot{Q}_\alpha$, such that
$q_\alpha$ forces $\dot{q}$ is the least upper bound. Let
$q_{\alpha+1} = q_\alpha \frown(\dot{q})$.
\end{proof}
\bigskip

Now suppose $p \in P_{\alpha_*}^0$. 
Note that $\mbox{supp}(p) \in [\alpha_*]^{<\theta}$. 

By a similar proof to the claim we can find, for each $n < \omega$, elements $\mathbf{q}_n(p) \in P^0_{\alpha_*}$with $\mathbf{q}_0(p)= p$, so that for all $n< \omega$:
\mn
\begin{itemize}
\item $\mathbf{q}_{n+1}(p) \geq \mathbf{q}_n(p)$;
\item For all $\alpha <\alpha_*$, $\mathbf{q}_{n+1}(p) \restriction_\alpha$ decides $\dot{F}_\alpha(\mathbf{q}_{n}(p)(\alpha))$.  (This is automatic whenever $\alpha \not \in \mbox{supp}(\mathbf{q}_{n})$, since then $P$ forces that $\dot{F}_\alpha(\mathbf{q}_n(p)(\alpha)) = \emptyset$.)
\end{itemize}

So we can choose $f_{n, \alpha,p} \in P_{\lambda \mu \sigma}$ such that each $\mathbf{q}_{n+1}(p) \restriction_\alpha$ forces that $\dot{F}_\alpha(\mathbf{q}_{n}(p)(\alpha)) = \check{f}_{n, \alpha,p}$.

Let $\mathbf{q}_\omega(p) \in P$ be the least upper bound of $(\mathbf{q}_n(p): n < \omega)$, which is possible by the claim. Let $P^0 = \{\mathbf{q}_\omega(p): p \in P^0_{\alpha_*}\}$. For each $q \in P^0$, choose $\mathbf{p}(q) \in P^0_{\alpha_*}$ such that $q = \mathbf{q}_\omega(\mathbf{p}(q))$. For each $n < \omega$, let $\mathbf{p}_n(q) = \mathbf{q}_{n}(\mathbf{p}(q))$, and for each $\alpha < \alpha_*$, let $f_{n, \alpha,q} = f_{n, \alpha, \mathbf{p}(q)}$.

Thus we have arranged that for all $q \in P^0$, $q$ is the least upper bound of $(\mathbf{p}_n(q): n < \omega)$, and for all $n < \omega$ and $\alpha < \alpha_*$, $\mathbf{p}_{n+1}(q) \restriction_\alpha$ forces that $\dot{F}_\alpha(\mathbf{p}_n(q)(\alpha)) = \check{f}_{n, \alpha}(q)$.

Write $X = \omega \times \alpha_* \times \lambda$. Choose $F: P^0 \to P_{X \mu \theta}$ so that for all $q, q'\in P^0$, if $F(q)$ and $F(q')$ are compatible, then for all $n < \omega$ and for all $\alpha <\alpha_*$, $f_{n, \alpha, q}$ and $f_{n, \alpha, q'}$ are compatible. For instance, let the domain of $F(q)$ be the set of all $(n, \alpha, \beta)$ such that $\beta$ is in the domain of $f_{n, \alpha, q}$, and let $F(q)(n, \alpha, \beta) = f_{n, \alpha, q}(\beta)$. 

Now suppose $(q_i: i < i_*)$ is a sequence from $P^0$ with $i_* < k$, such that $(F(q_i): i < i_*)$ are compatible.  Write $\Gamma = \bigcup_{i < i_*, n < \omega} \mbox{supp}(\mathbf{p}_{n}(q_i))$. 

By induction on $\alpha \leq \alpha_*$, we construct a least upper bound $s_\alpha$ to $(\mathbf{p}_{n}(q_i) \restriction_\alpha: i < i_*, n < \omega)$ in $P_\alpha$, such that $\mbox{supp}(s_\alpha) \subseteq \Gamma \cap \alpha$, and for $\alpha < \alpha'$, $s_{\alpha'} \restriction_\alpha = s_\alpha$. 

Limit stages of the induction are clear. So suppose we have constructed $s_\alpha$. If $\alpha \not \in \Gamma$ clearly we can let $s_{\alpha+1} = s_\alpha \,^\frown (0^{\dot{Q}_\alpha})$; so suppose instead $\alpha \in \Gamma$. Let $n < \omega$ be given. Then $(f_{n, \alpha, q_i}: i< i_*)$ are compatible, and $s_\alpha$ forces that $\dot{F}_\alpha(\mathbf{p}_{n}(q_i)(\alpha)) = \check{f}_{n, \alpha, q_i}$ for each $i < i_*$, since $\mathbf{p}_{n+1}(q_i) \restriction_\alpha$ does. Thus $s_\alpha$ forces that $(\mathbf{p}_n(q_i)(\alpha): i < i_*)$ has a least upper bound $\dot{r}_n$. Now $s_\alpha$ forces that $(\dot{r}_{n}: n < \omega)$ is an ascending chain in $\dot{Q}_\alpha$, so let $\dot{q}$ be such that $s_\alpha$ forces $\dot{q}$ is a least upper bound to $(\dot{r}_n: n < \omega)$. Let $s_{\alpha+1} = s_\alpha \,^\frown(\dot{q})$.

Thus the induction goes through, and $s_{\alpha_*}$ is a least upper bound $(q_i: i < i_*)$.
\end{proof}

The following class of forcing axioms, for $k = 3$, is related to Shelah's Ax$\mu_0$ from \cite{ShelahForcing} although the formulation is different. 

\begin{definition}\label{DefOfAxioms}
Suppose $\aleph_0 < \theta = \theta^{<\theta} \leq \lambda$, and suppose $3 \leq k < \omega$. Then say that Ax$(<k, \theta, \lambda)$ holds if for every forcing notion $P$ such that $|P| \leq \lambda$ and $P$ has the $(<k, \theta, \theta)$-amalgamation property, if $(D_\alpha: \alpha < \lambda)$ is a sequence of dense subsets of $P$, then there is an ideal of $P$ meeting each $D_\alpha$. (By  dense, we mean upwards dense: for every $p \in P$, there is $q \in D_\alpha$ with $q \geq p$.) Say that Ax$(<k, \theta)$ holds iff Ax$(<k, \theta, \lambda)$ holds for all $\lambda < 2^\theta$.
\end{definition}

By a typical downward Lowenheim-Skolem argument we could drop the condition that $|P| \leq \lambda$ in Ax$(k, \theta, \lambda)$, but we won't need this. Finally, note that Ax$(k, \theta, \lambda)$ implies that $2^\theta > \lambda$, since $P_{\theta 2 \theta}$ has the $(<k, \theta, \theta)$-amalgamation property and there is a family $2^\theta$ dense sets such that no ideal meets them all.

\begin{theorem}
\label{ForcingAxiomsCons}
Suppose $\aleph_0 < \theta$ are cardinals such that $\theta$ is regular and $\theta = \theta^{<\theta}$, and suppose $3 \leq k \leq \theta$. Suppose $\kappa \geq \theta$ has $\kappa^{<\kappa} = \kappa$. Then there is a forcing notion $P$ with the $(<k, \theta, \theta)$-amalgamation property (in particular, $\theta$-closed and $\theta^+$-c.c.), such that $P$ forces that Ax$(<k, \theta)$ holds and that $2^\theta = \kappa$. We can arrange $|P| =  \kappa$.
\end{theorem}

\begin{proof}
The proof is very similar to the proof of the consistency of Martin's axiom, see Theorem 16.13 of \cite{Jech}.

Let $(P_\alpha: \alpha \leq \kappa),(\dot{Q}_\alpha: \alpha < \kappa)$ be a $<\theta$-support iteration, such that (viewing $P_\alpha$-names as $P_\beta$-names in the natural way, for $\alpha \leq \beta < \kappa$): 
\mn
\begin{itemize}
\item Each $P_\alpha$ forces that $\dot{Q}_\alpha$ has the $(<k, \theta,
  \theta)$-amalgamation property;
\sn
\item Whenever $\alpha < \kappa$, and $\dot{Q}$ is a $P_\alpha$-name such that $|\dot{Q}| < \kappa$ and $P_\alpha$ forces $\dot{Q}$ has the $(<k, \theta, \theta)$-amalgamation property, then there is some $\beta \geq \alpha$ such that $P_\beta$ forces that $\dot{Q}_\beta$ is isomorphic to $\dot{Q}$;
\item Each $|P_\alpha| \leq \kappa$.
\end{itemize}
\mn
This is possible by the $\theta^+$-c.c., and using Lemma~\ref{ObviousLemmasForcingAP}(2). The point is that at each stage $\alpha$, if $P_\alpha$ forces that $|\dot{Q}| = \lambda < \kappa$, then we can choose a $P_\alpha$-name $\dot{Q}'$ such that $P_\alpha$-forces $\dot{Q} \cong \dot{Q}'$ and that $\dot{Q}'$ has universe $\lambda$; then there are only $|P_\alpha|^{\theta \cdot \lambda} \leq \kappa$-many possibilities for $\dot{Q}'$, up to $P_\alpha$-equivalence. Thus we can eventually deal with all of them.

Note that by $\kappa = \kappa^{<\kappa}$ we have in particular that $P_\kappa$ forces $2^\theta \leq \kappa$. Once we verify $P_\kappa$ forces that Ax$(<k, \theta, \lambda)$ for all $\lambda < \kappa$, it follows that $P_\kappa$ forces $2^\theta = \kappa$.

By the $\theta^+$-c.c., we have that whenever $\dot{X}$ is a $P_\kappa$-name for a subset of $\lambda$ for some $\lambda< \kappa$, then for some $\alpha < \kappa$ and some $P_\alpha$-name $\dot{X}_\alpha$ we have that $\dot{X}$ is forced to be equal to $\dot{X}_\alpha$.

Let $\mathbb{V}[G_\kappa]$ be a $P_\kappa$-generic extension of $\mathbb{V}$; for $\alpha < \kappa$ let $G_\alpha$ be the associated $P_\alpha$-generic extension of $\mathbb{V}$. Rephrasing the previous paragraph, we have that whenever $\lambda < \kappa$ and $X \subseteq \lambda$ is in $\mathbb{V}[G_\kappa]$, we have $X \in \mathbb{V}[G_\alpha]$ for some $\alpha < \kappa$.

Let $Q$ be a forcing notion in $\mathbb{V}[G_\kappa]$  with the $(<k, \theta, \theta)$-amalgamation property, with $|Q| < \kappa$; let $F: (Q, Q_0) \to_k P_{X \theta \theta}$ witness this, where we can suppose $X = \lambda < \kappa$. Let $\mathcal{D} = \{D_\alpha: \alpha < \lambda'\}$ be a set of dense subsets of $Q$ where $\lambda' < \kappa$. By the preceding, we can find $\alpha < \kappa$ such that $(Q, Q_0, F, \mathcal{D}) \in \mathbb{V}[G_\alpha]$. Thus we can find $P_\alpha$-names for them, $\dot{Q}, \dot{Q}_0, \dot{F}, \dot{\mathcal{D}}$. We have that $P_\alpha$ forces $\dot{Q}$ has the $(<k, \theta, \theta)$-amalgamation property. Then we can find some $\beta \geq \alpha$ such that it is forced $\dot{Q} \cong \dot{Q}_\beta$. Then in $\mathbb{V}[G_\kappa]$, if we let $H$ be the $\mathbb{V}[G_\beta]$-generic subset of $Q$ added by $\dot{Q}_\beta$, then this is an ideal of $Q$ meeting each dense set in $\mathcal{D}$, thus verifying Ax$(<k, \theta)$.
\end{proof}

\noindent
We now relate this to model theory.
\begin{definition}
% \label{}
Suppose $(\theta, \lambda)$ is a nice pair, and $\theta \leq \mu \leq \lambda$, and $T$ is simple. Then say that $T$ has $(<k, \lambda, \mu, \theta)$-type amalgamation if whenever $M \models T$ has size $\leq \lambda$, and whenever $M_* \preceq M$ is countable, then $\Gamma^\theta_{M, M_*}$ has the $(<k, \mu, \theta)$-amalgamation property, or equivalently, $\Gamma^\theta_{M, M_*} \rightarrow_k^w P_{X \mu \theta}$ for some set $X$.
\end{definition}

\noindent

\begin{lemma}
\label{PresOfFail}
Suppose $T$ fails $(<k, \lambda, \mu, \theta)$-type amalgamation, and $P$ has the $(<k, \mu, \theta)$-amalgamation property. Then $P$ forces that $\check{T}$ fails $(<k, \lambda, |\mu|, \theta)$-type amalgamation.
\end{lemma}

\begin{proof}
It suffices to show that if $Q$ is a forcing notion and $P$ forces
that $\check{Q} \rightarrow_k^w \check{P}_{\check{X} \mu \theta}$,
then $Q \rightarrow_k^w P_{X' \mu \theta}$ for some $X'$, by
Lemma~\ref{ObviousLemmasForcingAP}(4). (We then apply this to $Q =
\Gamma^\theta_{M, M_*}$ 
witnessing the failure of $(<k, \mu, \theta)$-amalgamation.)

Choose some $F_*: (P, P_0) \to_k P_{X_* \mu \theta}$, and let $\dot{G}$ be a $P$-name so that $P$ forces $\dot{g}: \check{Q} \rightarrow_k^w P_{\check{X} \mu \theta}$. For every $q \in Q$, choose $\mathbf{p}(q) \in P_0$ such that $\mathbf{p}(q)$ decides $\dot{G}(\check{q})$, say $\mathbf{p}(q)$ forces that $\dot{G}(\check{q}) = f(q)$. Let $X$ be the disjoint union of $X_*$ and $X$, and choose $F: Q \to P_{X \mu \theta}$ so that if $F(q)$ and $F(q')$ are compatible, then $f(q)$ and $f(q')$ are compatible, and $F_*(\mathbf{p}(q))$ and $F_*(\mathbf{p}(q'))$ are compatible.

Suppose $(q_i: i < i_*)$ is a sequence from $Q$ with $(F(q_i): i < i_*)$ compatible in $P_{X \mu \theta}$. Then $(F_*(\mathbf{p}(q_i)): i < i_*)$ are all compatible in $P_{X_* \mu \theta}$, so $(\mathbf{p}(q_i): i < i_*)$ are compatible in $P_0$ with the least upper bound $p$. Then $p$ forces each $\dot{F}(\check{q}_i) = f(q_i)$. But also (by choice of $F$), $(f(q_i): i < i_*)$ are compatible in $P_{Y, \mu, \theta}$, so $p$ forces that $(\check{q}_i: i < i_*)$ is compatible in $\check{Q}$, i.e. $(q_i: i < i_*)$ is compatible in $Q$.
\end{proof}

\begin{theorem}
\label{SatPortionPt2}
Suppose $T$ simple, and $\aleph_0 < \theta = \theta^{<\theta} \leq \lambda = \lambda^{\aleph_0}$, and Ax$(<k, \theta)$ holds. Suppose $2^\theta > \lambda^{<\theta}$, and suppose $3 \leq k \leq \aleph_0$. Then the following are equivalent:
\mn
\begin{itemize}
\item[(A)] $T$ has $(<k, \lambda, \theta, \theta)$-type amalgamation;
\item[(B)] $SP^1_T(\lambda, \theta, \theta)$ holds.
\end{itemize}
\end{theorem}

\begin{proof}
(B) implies (A): suppose (B) holds and $M \models T$ has size $\lambda$ and $M_* \preceq M$ is countable. Let $(p_\alpha(x): \alpha < \theta)$ be as in the definition of $SP^1_T(\lambda, \theta, \theta)$. Let $X = \{x\}$ be a singleton. Then $\Gamma^{\theta}_{M, M_*} \rightarrow^w_k P_{X \theta \theta}$, namely send $p(x) \in \Gamma^{\theta}_{M, M_*}$ to $\{(x, \alpha)\}$ for some $\alpha$ with $p(x) \subseteq p_\alpha(x)$.

(A) implies (B): let $M \models T$ have size at most $\lambda$ and let $M_* \preceq M$ be countable. Let $P$ be the $<\theta$-support product of $\theta$-many copies of $\Gamma^{\theta}_{M, M_*}$; then $P$ has the $(<k, \theta, \theta)$-amalgamation property and $|P| \leq \lambda^{<\theta}$. For each $p(x) \in \Gamma^\theta_{M, M_*}$ let $D_p$ be the dense subset of $P$ consisting of all $f \in P$ such that for some $\gamma \in \mbox{dom}(f)$, $f(\gamma)$ extends $p(x)$. By Ax$(<k,\lambda^{<\theta}, \theta)$ we can choose an ideal $I$ of $P$ meeting each $D_p$. This induces a sequence $(p_\gamma(x): \gamma < \theta)$ of  partial types over $M$ that do not fork over $M_*$, such that for all $p(x) \in \Gamma^\theta_{M, M_*}$ there is $\gamma < \theta$ with $p(x) \subseteq p_\gamma(x)$. To finish, extend each $p_\gamma(x)$ to a complete type over $M$ not forking over $M_*$.
\end{proof}
\newpage

\section{Non-Forking Diagrams}\label{IndSystems}
\bigskip

Suppose $T$ is a simple theory in a countable language. We wish to study various type amalgamation properties of $T$; in particular we will be looking at systems of types $(p_s(x): s \in P)$ over a system of models $(M_s: s \in P)$, for some $P \subseteq \mathcal{P}(I)$ closed under subsets. For this to be interesting, we need $(M_s: s \in P)$ to be independent in a suitable sense, which we define in this section.

The following definition is similar to the first author's definition of independence in \cite{ShelahIso} in the context of stable theories, see Section XII.2. In fact we are modeling our definition after Fact 2.5 there (we cannot take the definition exactly from \cite{ShelahIso} because we allow $P$ to contain infinite subsets of $I$).

\begin{definition}
Let $T$ be simple.

Suppose $I$ is an index set and $P \subseteq \mathcal{P}(I)$ is downward closed. Say that $(A_s: s \in P)$ is a diagram (of subsets of $\mathfrak{C}$) if each $A_s \subseteq \mathfrak{C}$ and $s \subseteq t$ implies $A_s \subseteq A_t$. Say that $(A_s: s \in P)$ is a non-forking diagram if for all $s_i: i < n, t_j: j < m \in P$, $\bigcup_{i < n} A_{s_i} \forkindep_{\bigcup_{i,j} A_{s_i \cap t_j}}\bigcup_{j < m} A_{t_j}$. Say that $(A_s: s \in P)$ is a continuous diagram if for every $X \subseteq P$, $\bigcap_{s \in X} A_s = A_{\bigcap X}$. (If $X$ is finite then this is a consequence of non-forking.)
\end{definition}

\noindent
Note that $(A_s: s \in P)$ is continuous if and only if for every $a \in \bigcup_{s \in P} A_s$, there is some least $s \in P$ with $a \in A_s$.  Also note that if $(A_s: s \in P)$ is non-forking (continuous) and $Q \subseteq P$ is downward closed then $(A_s: s \in Q)$ is non-forking (continuous).
\begin{lemma}
\label{EquivsForInd}
Suppose $(A_s: s \in P)$ is a 
diagram of subsets of $\mathfrak{C}$. Then the following are equivalent:
\mn
\begin{itemize}
\item[(A)] For all downward-closed subsets $S, T \subseteq P$, $\bigcup_{s \in S} A_s \forkindep_{\bigcup_{s \in S \cap T} A_s} \bigcup_{t \in T} A_t$.
\item[(B)] $(A_s: s \in P)$ is non-forking.
\end{itemize}
\end{lemma}

\begin{proof}
(A) implies (B) is trivial. 

(B) implies (A): we proceed by induction on $\kappa$ to show that for all $s_\alpha: \alpha < \kappa$, $t_\beta: \beta < \kappa$, $\bigcup_{\alpha} A_{s_\alpha} \forkindep_{\bigcup_{\alpha \beta} A_{s_\alpha \cap t_\beta}} \bigcup_\beta A_{t_\beta}$. This suffices to prove (A) since when $\kappa \geq |S| + |T|$ then $s_\alpha, t_\beta$ can just enumerate $S$ and $T$, in which case $s_\alpha \cap t_\beta$ enumerates $S \cap T$. For the induction, when $\kappa$ is finite use the hypothesis (B), and when $\kappa$ is infinite use the local character of nonforking.
\end{proof}

\noindent
The following lemma is similar to Lemma 2.3 from \cite{ShelahIso} Section XII.2.
\begin{lemma}
\label{SuffCondForInd}
Suppose $P \subseteq \mathcal{P}(I)$ is downward closed and $(A_s: s \in P)$ is a continuous diagram of subsets of $\mathfrak{C}$. Suppose there is a well-ordering $<_*$ of $\bigcup_s A_s$ such that for all $a \in \bigcup_s A_s$, $a$ is free from $\{b \in \bigcup_s A_s: b <_* a\}$ over $\{b \in s_a: b <_* a\}$, where $s_a$ is the least element of $P$ with $a \in A_{s_a}$. Then $(A_s: s \in P)$ is non-forking.
\end{lemma}

\begin{proof}
Let $(a_\alpha:\alpha < \alpha_*)$ be the $<_*$-increasing enumeration
of $\bigcup_s A_s$, and let $s_\alpha$ be the least element of $P$
with $a_\alpha \in A_{s_\alpha}$. For each $\alpha \leq \alpha_*$ and
for each $s \in P$ let $A_{s, \alpha} = A_s \cap \{a_\beta: \beta <
\alpha\}$. We show by induction on $\alpha$ that $(A_{s,\alpha}: s \in
P)$ is non-forking. 

Limit stages are clear. So suppose we have shown $(A_{s, \alpha}:s \in P)$ is non-forking. Let $(s_i: i < n), (t_j: j < m) \in P$ be given. We wish to show \\$\bigcup_{i < n} A_{s_i, \alpha+1} \forkindep_{\bigcup_{i < n, j < n} A_{s_i \cap t_j, \alpha+1}}\bigcup_{j < m}A_{t_j ,\alpha+1}$. Write $A = \bigcup_{i < n} A_{s_i, \alpha}$, write $B = \bigcup_{j < m} A_{t_j, \alpha}$, and write $C = \bigcup_{i,j} A_{s_i \cap t_j, \alpha}$. Define $A', B', C'$ similarly except with $\alpha+1$ replacing $\alpha$. We are trying to show $A' \forkindep_{C'} B'$, and by the inductive hypothesis, $A \forkindep_C B$, and we also know $a_\alpha \forkindep_{A_{s_\alpha, \alpha}} AB$. 

If $a_\alpha \not \in s_i$ and $a_\alpha \not \in t_j$ for any $i, j$ then $A' = A, B'=B, C' = C$ and so we are done. If $a_\alpha \in s_i \cap t_j$ and hence $A_{s_\alpha, \alpha} \subseteq C$, then $A' = A a_\alpha, B' = B a_\alpha, C' = C a_\alpha$ and $a_\alpha \forkindep_C AB$ (by monotonicity), so we are done by $A \forkindep_C B$ and transitivity.

Up to symmetry, the final case is $a_\alpha \in s_i$ but is not in any $t_j$. Then $A_{s_\alpha, \alpha} \subseteq A$ and $A' = A a_\alpha, B' = B, C' =C$. By monotonicity we have $a_\alpha \forkindep_A AB$, so $a_\alpha A \forkindep_A B$, so by transitivity $a_\alpha A \forkindep_C B$ as desired.
\end{proof}

\begin{theorem}
\label{IndependentSystemsThm}
Suppose $T$ is a simple theory in a countable language, and suppose $\mathbf{A}$ is a set of cardinality $\lambda$, where $\lambda = \lambda^{\aleph_0}$. Then we can find a continuous, non-forking diagram of models $(M_s: s \in [\lambda]^{\leq \aleph_0})$ such that $\mathbf{A} \subseteq \bigcup_s M_s$, and such that for all $S \subseteq \lambda$, $\bigcup_{s \in [S]^{\leq \aleph_0}} M_s$ has size at most $|S| \cdot \aleph_0$.
\end{theorem}

\begin{proof}
Enumerate $\mathbf{A} = (a_\alpha: \alpha < \lambda)$. 

We define $(\mbox{cl}({\{\alpha\}}): \alpha < \lambda)$ inductively as follows, where each $\mbox{cl}({\{\alpha\}})$ is a countable subset $\alpha+1$ with $\alpha \in \mbox{cl}({\{\alpha\}})$. Suppose we have defined $(\mbox{cl}({\{\beta\}}): \beta < \alpha)$. Choose a countable set $\Gamma \subseteq \alpha$ such that $a_\alpha \forkindep_{\{a_\beta: \beta \in \Gamma\}}\{a_\beta: \beta < \alpha\}$; put $\mbox{cl}({\{\alpha\}})= \{\alpha\} \cup \bigcup_{\beta \in \Gamma} \mbox{cl}({\{\beta\}})$. 

Now, for each $s \subseteq \lambda$, let $\mbox{cl}(s) := \bigcup_{\alpha \in s} \mbox{cl}(\{\alpha\})$. Say that $A \subseteq \lambda$ is closed if $\mbox{cl}(A) = A$; this satisfies the usual properties of a set-theoretic closure operation, that is $\mbox{cl}(A) \supseteq A$, and $A \subseteq B$ implies $\mbox{cl}(A) \subseteq \mbox{cl}(B)$, and $\mbox{cl}^2(A) = \mbox{cl}(A)$, and $\mbox{cl}$ is finitary: in fact $\mbox{cl}(A) = \bigcup_{\alpha \in A} \mbox{cl}(\{\alpha\})$, which is even stronger. Finally, $|\mbox{cl}(A)| \leq |A| + \aleph_0$.

For each $s \in [\lambda]^{\leq \omega}$, let $A_s = \{a_\alpha: \alpha < \lambda \mbox{ and } \mbox{cl}(\{\alpha\}) \subseteq s\}$.  Since each $a_\alpha \in A_{\mbox{cl}(\{\alpha\})}$, we have $\bigcup_s A_s = \mathbf{A}$. Further, $(A_s: s \in [\lambda]^{\leq \omega})$ is clearly a continuous diagram of sets; we claim that $(A_s: s \in [\lambda]^{\leq \omega})$ is a non-forking diagram of sets. But this follows from Lemma~\ref{SuffCondForInd}, since each $a_\alpha \forkindep_{A_{\mbox{cl}(\{\alpha\})} \cap \{a_\beta: \beta < \alpha\}} \{a_\beta: \beta < \alpha\}$.

For each $\alpha \leq \lambda$, and each $u \in [\lambda]^{<\omega}$ let $\mathcal{A}_{\alpha, u} =
\{\mbox{cl}(s \cup u) \cap \alpha: s \in [\alpha]^{<\omega}\}$. We show by induction on
$\alpha \leq \lambda$ that for all $u \in [\lambda]^{<\omega}$, $(\mathcal{A}_{\alpha, u}, \subset)$ is
well-founded. There will be separate step case and limit case.

Suppose we have shown $(\mathcal{A}_{\alpha, u}, \subset)$ is well-founded. Then $\mathcal{A}_{\alpha+1, u} = X_0 \cup X_1$ where $X_0 = \{\mbox{cl}(s \cup u) \cap (\alpha+1): s \in [\alpha]^{<\omega}\}$ and $X_1 = \{\mbox{cl}(s \cup \{\alpha\} \cup u) \cap (\alpha+1): s \in [\alpha]^{<\omega}\}$. It suffices to show each of $X_0, X_1$ is well-founded under subset. Write $v_0 = u$ and $v_1 = u \cup \{\alpha\}$. Then $X_i = \{\mbox{cl}(s \cup v_i) \cap (\alpha+1): s \in [\alpha]^{<\omega}\}$. 

By the induction hypothesis, $\mathcal{A}_{\alpha, v_i}$ is well-founded under subset, so it suffices to show that $\mbox{cl}(s\cup v_i) \cap (\alpha+1) \subseteq \mbox{cl}(t \cup v_i) \cap (\alpha+1)$ if and only if $\mbox{cl}(s \cup v_i) \cap \alpha \subseteq \mbox{cl}(t \cup v_i) \cap \alpha$, for all $s, t \in [\alpha]^{<\omega}$. It suffices to show that $\alpha \in \mbox{cl}(s \cup v_i)$ if and only if $\alpha \in \mbox{cl}(t \cup v_i)$, but both are equivalent to $\alpha \in \mbox{cl}(v_i)$.

Now suppose we have shown $(\mathcal{A}_{\alpha, u}, \subseteq)$ is well-founded for all $u$ and for all $\alpha < \delta$ where $\delta$ is a limit. Suppose towards a contradiction $(\mathcal{A}_{\delta, u}, \subseteq)$ were not well-founded, say it had the infinite descending chain $\mbox{cl}(s_n \cup u) \cap \delta$. Choose $\alpha < \delta$ with $s_0 \in [\alpha]^{<\omega}$. Then we have each $s_n \subseteq \mbox{cl}(s_0 \cup u)$; thus $s_n \subseteq \alpha \cup \mbox{cl}(u)$. Put $t_n = s_n \cap \alpha$; then each $\mbox{cl}(s_n \cup u) \cap \delta = \mbox{cl}(t_n \cup u) \cap \delta$. But then $(\mbox{cl}(t_n \cup u) \cap \alpha: n < \omega)$ must be strictly descending, since $(\mbox{cl}(t_n \cup u) \cap \delta: n < \omega)$ is, and each $\mbox{cl}(t_n \cup u) \cap (\delta \backslash \alpha) = \mbox{cl}(u) \cap (\delta \backslash \alpha)$.

% ; note that $\mbox{cl}(X)
% = X$. Now suppose $s, t \in [\alpha]^{<\omega}$.  We claim that
% $\mbox{cl}(s \cup \{\alpha\}) \subseteq  \mbox{cl}(t \cup \{\alpha\})$
% iff $\mbox{cl}(s \cup X) \subseteq \mbox{cl}({t \cup X})$. But this is
% clear, since $\mbox{cl}(s \cup \{\alpha\}) = \mbox{cl}(s) \cup X \cup
% \{\alpha\}$, and $\mbox{cl}(t \cup \{\alpha\}) = \mbox{cl}(t) \cup X
% \cup \{\alpha\}$, and $\mbox{cl}(s \cup X) = \mbox{cl}(s) \cup X$, and 
% $\mbox{cl}(t \cup X) = \mbox{cl}(t) \cup X$.

% By the induction hypothesis, $(\mathcal{A}_\alpha, \subseteq)$ is well-founded. But then also $(\{\mbox{cl}(s \cup X): s \in [\alpha]^{<\aleph_0})$ is well-founded,

Hence $\mathcal{A} := \mathcal{A}_{\lambda, \emptyset} = \{\mbox{cl}(s): s \in [\lambda]^{<\aleph_0}\}$ is well-founded under subset. Note that for all $s \in \mathcal{A}$, since $\mbox{cl}(s) = s$ we have $A_s = \{a_\alpha: \alpha \in s\}$.

Let $<_*$ be a well-order of $\mathcal{A}$ refining $\subset$. Now by induction on $<_*$, choose countable models $(M_s: s \in \mathcal{A})$ so that $M_s \supseteq A_s$ and $M_s \supseteq M_t$ for $t \subseteq s$ and such that $M_s \forkindep_{A_s \cup \bigcup \{M_t: t \in \mathcal{A}, t \subset s\}}\mathbf{A} \cup \bigcup\{M_t:t  \in \mathcal{A}, t <_* s\}$. Finally, given $s \in [\lambda]^{\leq \omega}$, let $M_s := \bigcup\{M_t: t \in \mathcal{A}, t \subseteq s\}$. This is a continuous diagram of models, and for all $S \subseteq \lambda$, $\{t \in \mathcal{A}: t \subseteq S\}$ has size at most $|S|\cdot \aleph_0$, so to finish the proof of the theorem it suffices to show $(M_s: s \in [\lambda]^{\leq \aleph_0})$ is non-forking. 

Enumerate $\mathcal{A} = (u_\alpha: \alpha < \alpha_*)$ in $<_*$-increasing order. For each $\alpha \leq \alpha_*$ let $(B^\alpha_{ s}: s \in [\lambda]^{\leq \aleph_0})$ be the continuous diagram of models defined via $B^\alpha_{ s} = A_s \cup \bigcup\{M_{u_\beta}: u_\beta \subseteq s, \beta < \alpha\}$. So $B^0_{ s} = A_s$, $B^{\alpha_*}_{s} = M_s$, and it suffices to show by induction on $\alpha$ that $(B^{\alpha}_{ s}: s \in [\lambda]^{\leq \aleph_0})$ is nonforking.

The base case and limit cases are clear. So suppose $(B^{\alpha}_{s}: s \in [\lambda]^{\leq \aleph_0})$ is nonforking; we try to show $(B^{\alpha +1}_{s}: s \in [\lambda]^{\leq \aleph_0})$ is nonforking. Let $s_i: i < n, t_j: j < m$ be from $[\lambda]^{\leq \aleph_0}$; we want to show $\bigcup_{i < n} B^{\alpha+1}_{s_i}\forkindep_{\bigcup_{i, j} B^{\alpha+1}_{s_i \cap t_j}}\bigcup_{j < m} B^{\alpha+1}_{t_j}$.

Write $A = \bigcup_{i < n} B^\alpha_{s_i}$, write $B = \bigcup_{j < m} B^\alpha_{s_j}$, and write $C = \bigcup_{i, j} B^\alpha_{s_i \cap s_j}$, and let $A', B', C'$ be the same but with $\alpha+1$. We know $A \forkindep_C B$ by the inductive hypothesis and we are trying to show $A' \forkindep_{C'} B'$. We also know, by construction of $M_{u_\alpha}$, that $M_{u_\alpha} \forkindep_{B^\alpha_{u_\alpha}} \bigcup\{B^\alpha_s: s \in [\lambda]^{\leq \aleph_0}\}$.

If $u_\alpha$ is not contained in any $s_i$ or $t_j$ then $A= A', B= B', C= C'$ and we are done. If $u_\alpha$ is contained in some $s_i \cap t_j$ then $A' = A M_{u_\alpha}$, $B' = B M_{u_\alpha}$, $C' = C M_{u_\alpha}$, and $M_{u_\alpha} \forkindep_C AB$, so we are done by transitivity. The remaining case (up to symmetry) is that $u_\alpha$ is contained in some $s_i$ but not in any $t_j$. Then $A' = A M_{u_\alpha}$, $B' = B, C' = C$, and $M_{u_\alpha} \forkindep_{A} B$, so we are again done by transitivity.
\end{proof}
\newpage

\section{Amalgamation properties}\label{ModelTheory}
\bigskip

Suppose $T$ is a simple theory in a countable language. We now explain what we mean by $T$ having $<k$-type amalgamation.
\begin{definition}
\label{DefOfTypeAP}
Given $\Lambda \subseteq \,^n m$, let $P_{\Lambda}$ be the set of all partial functions from $n$ to $m$ which can be extended to an element of $\Lambda$; so $P_{\Lambda}$ is a downward-closed subset of $\mathcal{P}(n \times m)$, and $\Lambda$ is the set of maximal elements of $P_{\Lambda}$. 

Suppose $(M_u: u \subseteq n)$ is a non-forking diagram of models. Then by a $(\Lambda, \overline{M})$-array, we mean a  non-forking diagram of models $(N_s: s \in P_{\Lambda})$, together with maps $(\pi_s: s \in P_{\Lambda})$ such that each each $\pi_s: M_{\mbox{dom}(s)} \cong N_s$, and such that $s \subseteq t$ implies $\pi_s \subseteq \pi_t$. 
\end{definition}

\begin{definition}
% \label{}
Suppose $\Lambda \subseteq \,^n m$. Then $T$ has $\Lambda$-type amalgamation if, whenever $(M_u: u \subseteq n)$ is a  non-forking diagram of models, and whenever $p(x)$ is a complete type over $M_n$ in a single variable which does not fork over $M_0$, and whenever $(N_s, \pi_s: s \in P_{\Lambda})$ is a $(\Lambda, \overline{M})$-array, then $\bigcup_{\eta \in \Lambda} \pi_\eta(p(x))$ does not fork over $N_0$. 

Suppose $3 \leq k \leq \aleph_0$; then say that $T$ has $< k$-type amalgamation if whenever $|\Lambda| < k$, then $T$ has $\Lambda$-type amalgamation. 
\end{definition}

\noindent
In the definition of $\Lambda$-type amalgamation, it would not matter if we required each $M_u$ to be countable, by a downward Lowenheim-Skolem argument.

\begin{example}
% \label{}
Every simple theory has $< 3$-type amalgamation. 
\end{example}
\begin{proof}
Suppose $\Lambda \subseteq \,^n m$ has $|\Lambda| = 2$ and $(M_u: u \subseteq n)$ is a non-forking diagram of models and $p(x)$ is a complete type over $M_n$ in a single variable which does not fork over $M_0$. Suppose $(N_s, \pi_s: s \in P_\Lambda)$ is a $(\Lambda, \overline{M})$-array. Write $\Lambda = \{\eta_0, \eta_1\}$. Write $K_i = \pi_{\eta_i}[M_n]$ for $i < 2$ and let $q_i = \pi_{\eta_i}(p(x))$. By the independence theorem for simple theories, $q_0(x) \cup q_1(x)$ does not fork over $K_0 \cap K_1$. But $K_0 \cap K_1 \subseteq K_0$ and since $q_0(x)$ does not fork $N_0$, also $q_0(x) \cup q_1(x)$ does not fork over $N_0$ by transitivity.
\end{proof}

\begin{example}
$T_{rg}$ has $<\aleph_0$-type amalgamation.
\end{example}
\begin{proof}
This follows from the fact that if $(A_s: s \in P)$ is any nonforking diagram of sets and $p_s(x) \in S^1(A_s)$ for each $s \in P$, if each $p_s(x)$ does not fork over $A_0$ and if $p_s(x) \subseteq p_t(x)$ for $s \subseteq t$, then $\bigcup_s p_s(x)$ is consistent and does not fork over $A_0$.
\end{proof}

\begin{example}
% \label{}
Suppose $\ell > k \geq 2$. Let $T_{\ell, k}$ be the theory of the
generic $k$-ary, $\ell$-clique free hypergraph; these examples were
introduced by Hrushovski \cite{Hrush}, where he proved they have quantifier elimination, and 
$T_{\ell, k}$ is simple if and only if $k \geq 3$. For $k \geq 3$ and $A\subseteq B, p(x) \in S(B)$, we have that $p$ forks over $A$ if and only if $p$ is realized in $B \backslash A$.

Then: for $k \geq 3$, $T_{\ell, k}$ has $<k$-type 
amalgamation but not $<k+1$-type amalgamation.
\end{example}

\begin{proof}
Let $R$ denote the edge relation of $T_{\ell, k}$.

First we show $T_{\ell, k}$ has $<k$-type amalgamation. Suppose $\Lambda \subseteq \,^{n} m$ with $|\Lambda| < k$, and
$(M_u: u \subseteq n)$ are given, and suppose $p(x)$ is a
complete type over $M_{n}$. Suppose towards a contradiction there were
a $(\Lambda, \overline{M})$-array $(N_s, \pi_s: s \in P_{\Lambda})$
with $\bigcup_{\eta \in \Lambda} \pi_\eta[p(x)]$
forking over $N_0$. Then we must have created some
$\ell$-clique $(x, a_i:  < \ell-1),$ where each $a_i \in N_\eta$ for some $\eta \in \Lambda$. That is, $a_i: i < \ell-1$ is an $R$-clique and for each $u \in [\ell-1]^{k-1}$ there is some $\eta \in \Lambda$ such that $\pi_\eta[p(x)]$ implies $R(x, a_i: i \in u)$.

For each $i < \ell-1$, let $h[\{i\}]$ be the least $s \in P_{\Lambda}$ with $a_i \in N_s$. For $u \subseteq \ell-1$ let $h[u] = \bigcup_{i \in u} h[\{i\}]$. The following must hold:
\mn
\begin{itemize}
\item[(I)]  For every $u \in [\ell-1]^{k-1}$, $h[u] \in P_{\Lambda}$, as some $\pi_\eta(p(x))$ implies $R(x, a_i: i \in u)$, and any such $\eta$ contains $h[u]$;
\sn
\item[(II)] $h[\ell-1] \not \in P_{\Lambda}$, as if $h[\ell-1] \subseteq \eta$ then $\pi_\eta(p(x))$ would contain $R(x, a_i: i \in u)$ for all $u \in [\ell-1]^{k-1}$ and so would be inconsistent.
\end{itemize}
\mn
By (II), for each each $\eta \in \Lambda$ we must have $h[\ell-1] \not
\subseteq \eta$; thus we can choose $i_\eta < \ell-1$ such that
$h[\{i_\eta\}] \not \subseteq \eta$. Let $u = \{i_\eta: \eta \in \Lambda\}
\in [\ell-1]^{<k}$. 
Clearly then $h[u] \not \in P_{\Lambda}$, but this contradicts (I).

Now we show that $T_{\ell, k}$ fails $<k+1$-type amalgamation. Indeed, let $\Lambda \subseteq \,^k 2$ be the set of all $f: k \to 2$ for which there is exactly one $i < k$ with $f(i) = 1$; so $|\Lambda| = k$. Also, let $(M_u: u \subseteq k)$ be a  non-forking diagram of models so that there are $a_i \in M_{\{i\}}$ for $i < k$ and there are $b_j \in M_0$ for $n < \ell-k-1$, such that every $k$-tuple of distinct elements from $(a_i, b_j: i < k, j < \ell-k-1)$ is in $R$ except for $(a_i: i < k)$. Let $p(x)$ be the partial type over $M_k$ which asserts that $R(x, \overline{a})$ holds for every $k-1$-tuple $\overline{a}$ of distinct elements from $(a_i, b_j: i < k, j < \ell-k-1)$.

It is not hard to find a $(\Lambda, \overline{M})$-array $(N_s, \pi_s: s \in P_{\Lambda})$ such that, if we write $\pi_{\{(i, 0)\}}(a_i) = c_i$, then $R(c_i: i < k)$ holds; but now we are done, since $\bigcup_{f \in \Lambda} \pi_f[p(x)]$ is inconsistent. 
\end{proof}

\noindent
The following is the key consequence of $<k$-type amalgamation.
\begin{theorem}
\label{SatPortion}
Suppose $T$ is a simple theory with $<k$-type amalgamation. Then for all nice pairs $(\theta, \lambda)$, $T$ has $(<k, \lambda, \theta, \theta)$-type amalgamation.
\end{theorem}

\begin{proof}
By Theorem~\ref{IndependentSystemsThm}, it suffices to show that if $(\mathbf{M}_s: s \in [\lambda]^{< \theta})$ is a continuous non-forking diagram of countable models such that each $|\mathbf{M}_s| < \theta$, then writing $\mathbf{M}= \bigcup_s \mathbf{M}_s$, we have that $\Gamma^\theta_{\mathbf{M}, \mathbf{M}_0} \rightarrow_k^w P_{X \theta \theta}$ for some $X$. Let $<_*$ be a well-ordering of $\mathbf{M}$.

Given $A \in [\mathbf{M}]^{<\theta}$ let $s_A$ be the $\subseteq$-minimal $s \in [\lambda]^{<\theta}$ with $A \subseteq M_{s_A}$, possible by continuity.

Let $P$ be the set of all $p(x) \in \Gamma^{\theta}_{\mathbf{M}, \mathbf{M}_0}$ such that for some $s \in [\lambda]^{<\theta}$, $p(x)$ is a complete type over $\mathbf{M}_{s}$; we write $p(x, \mathbf{M}_s)$ to indicate this. $P$ is dense in $\Gamma^{\theta}_{\mathbf{M}, \mathbf{M}_0}$, so it suffices to show that $P \rightarrow_k^w P_{\lambda \theta \theta}$ for some $\lambda$.

\begin{claim} For some set $X$, we can find $F: P \to P_{X \theta \theta}$ so that if $F(p(x, \mathbf{M}_s))$ is compatible with $F(q(x, \mathbf{M}_t))$, then:
\mn
\begin{itemize}
\item $s$ and $t$ have the same order-type, and if we let $\rho: s \to t$ be the unique order-preserving bijection, then $\rho$ is the identity on $s\cap t$;
\item $\mathbf{M}_s$ and $\mathbf{M}_t$ have the same
  $<_*$-order-type, and the unique $<_*$-preserving bijection from
  $\mathbf{M}_s$ to $\mathbf{M}_t$ is in fact an isomorphism $\tau:
  \mathbf{M}_s \cong \mathbf{M}_t$ which is the identity on $\mathbf{M}_{s \cap t}$;
\sn
\item For each finite $\overline{a} \in \mathbf{M}_{s}^{<\omega}$, if we write $s' = s_{\overline{a}}$ and if we write $t' = s_{\tau(\overline{a})}$, then: $\rho[s'] = t'$ and $\tau \restriction_{\mathbf{M}_{s'}}: \mathbf{M}_{s'} \cong \mathbf{M}_{t'}$.
\sn
\item $\tau[p(x)] = q(x)$.
\end{itemize}
\end{claim}
\begin{proof}
Given $p(x, \mathbf{M}_s) \in P$ let $\mathcal{D}_p$ consist of the following data: 
\begin{itemize}
\item The order-type of $s$, call it $\gamma$; let $\rho_0: \gamma \to s$ be the order-preserving bijection
\item The $<_*$-order-type of $\mathbf{M}_s$, call it $\delta$; let $\tau_0: \delta \to \mathbf{M}_s$ be the order-preserving bijection;
\item The structure $N$ with universe $\delta$, such that $\tau_0$ is an isomorphism from $N$ to $\mathbf{M}_s$;
\item $\tau_0^{-1}(p(x))$;
\item The set of all $(\overline{a}, \alpha)$ such that $\overline{a} \in N$ and $\alpha < \gamma$ and $\rho_0(\alpha) \in s_{\tau_0(\overline{a})}$.
\end{itemize}

Note that there are only $2^{<\theta}$ possibilities for $\mathcal{D}_p$. Thus, letting $X_0$ be any set of cardinality $\theta$, by choosing an antichain in $P_{X_0 \theta \theta}$ of size $2^{<\theta}$, we can find $F_0: P \to P_{X_0 \theta \theta}$ so that for all $p, q \in P$, if $F_0(p)$ and $F_0(q)$ are compatible then $\mathcal{D}_p = \mathcal{D}_q$.

Let $F_1: P \to P_{\lambda \theta \theta}$ send $p(x, \mathbf{M}_s)$ to the function with domain $s$ sending $\alpha \in s$ to its order-type in $s$. Let $F_2: P \to P_{\mathbf{M} \theta \theta}$ send $p(x, \mathbf{M}_s)$ to the function with domain $\mathbf{M}_s$ sending $a \in \mathbf{M}_s$ to its $<_*$-order-type in $\mathbf{M}_s$. Let $X = X_0 \cup \lambda \cup \mathbf{M}$ and let $F(p) = F_0(p) \cup F_1(p) \cup F_2(p)$. A straightforward verification shows that this works.
\end{proof}

Fix $F, X$ as in the claim. Note that it follows that for every $s' \subseteq s$, $\rho \restriction_{\mathbf{M}_{s'}}: \mathbf{M}_{s'} \cong \mathbf{M}_{\rho[s']}$, since $\mathbf{M}_{s'} = \bigcup\{\mathbf{M}_{s_{\overline{a}}}: \overline{a} \in (\mathbf{M}_{s'})^{<\omega}\}$ and similarly for $\mathbf{M}_{t'}$.

We claim that $F$ works. So suppose $p_i(x, \mathbf{M}_{s_i}): i < i_*$ is a sequence from $P$ for $i_* < k$, such that $(F(p_i(x)): i < i_*)$ is compatible in $P_{x \theta \theta}$.

Let $\gamma_*$ be the order-type of some or any $s_i$. Enumerate $s_i = \{s_i(\gamma): \gamma < \gamma_*\}$ in increasing order, and for $u \subseteq \gamma_*$ let $s_i[u] = \{s_i(\gamma): \gamma \in u\}$. Let $E$ be the equivalence relation on $\gamma_*$ defined by: $\gamma E \gamma'$ iff for all $i, i' < k$, $s_i(\gamma) = s_{i'}(\gamma)$ iff $s_i(\gamma') = s_{i'}(\gamma')$. Then $E$ has finitely many classes; enumerate them as $(u_j: j < n)$.  Then $s_i$ is the disjoint union of $s_i[u_j]$ for $j < n$. Moreover, $s_i[u_j] \cap s_{i'}[u_{j'}] = \emptyset$ unless $j = j'$; and if $s_{i}[u_j] \cap s_{i'}[u_j] \not= \emptyset$ then $s_i[u_j] = s_{i'}[u_j]$. For each $j < n$, enumerate $\{s_i[u_j]: i < i_*\} = (Y_{\ell, j}: \ell < m_j) $ without repetitions. Let $m = \max(m_j: j < n)$; and for each $i <i_*$, define $\eta_i \in \,^n m$ via: $\eta_i(j) =$ the unique $\ell < m_i$ with $s_{i}[u_j] = Y_{\ell, j}$. 

Let $\Lambda = \{\eta_i: i < i_*\}$. For each $s \in P_{\Lambda}$, let $N_{s} = \mathbf{M}_{t_s}$ where $t_s = \bigcup_{(j, \ell) \in s} Y_{\ell, j}$. Also,  define $(M_u: u \subseteq n) := (N_{\eta_0 \restriction_u}: u \subseteq n)$. Then the hypotheses on $F$ give commuting isomorphisms $\pi_s: M_{\mbox{dom}(s)} \cong N_s$ for each $s \in P_{\Lambda}$, in such a way that $(\overline{N}, \overline{\pi})$ is a $(\lambda, \overline{M})$-array, and each $\pi_{\eta_i}(p_0(x)) = p_i(x)$. It follows by hypothesis on $T$ that $\bigcup_{i < i_*} p_i(x)$ does not fork over $N_0$, as desired.

\end{proof}

\begin{corollary}
\label{MinClassSuff}
Suppose $T$ is simple, with $<\aleph_0$-type amalgamation.
\mn
\begin{itemize}
\item[(A)] Suppose $\theta$ is a regular uncountable cardinal. Then for any $M \models T$ and any $M_0 \preceq M$ countable, $\Gamma^{\theta}_{M, M_0}$ has the $(<\aleph_0, \theta, \theta)$-amalgamation property.
\item[(B)]Suppose $(\theta, \lambda)$ is a nice pair, and suppose that $\theta \leq \mu \leq \lambda$ satisfies $\mu = \mu^{<\theta}$ and $2^\mu \geq \lambda$. Then $SP^1_T(\lambda, \mu, \theta)$ holds.
\item[(C)] If the singular cardinals hypothesis holds, then $T \leq_{SP} T_{rg}$.
\end{itemize}
\end{corollary}

\begin{proof}
(A) follows immediately from Theorem~\ref{SatPortion}. 

(B): Suppose $M \models T$ has $|M| \leq \lambda$, and suppose $M_0 \preceq M$ is countable. Choose some $F: \Gamma^\theta_{M, M_0} \to_{\aleph_0}^w P_{\lambda \theta \theta}$. By Corollary~\ref{ER}, we can find $(\mathbf{f}_\gamma: \gamma < \mu)$ such that whenever $f \in P_{\lambda \theta \theta}$ then $f \subseteq \mathbf{f}_\gamma$ for some $\gamma < \mu$; for each $\gamma < \mu$, choose $q_\gamma(x)$, a complete type over $M$ not forking over $M_0$, and extending $\bigcup\{p(x): F(p(x)) \subseteq \mathbf{f}_\gamma\}$. Then clearly $(q_\gamma(x): \gamma < \mu)$ witnesses $SP^1_T(\lambda, \mu, \theta)$.

(C): Suppose $(\theta, \lambda)$ is a nice pair, and $SP_{T_{rg}}(\lambda, \theta)$ holds; we want to show $SP_T(\lambda, \theta)$ holds. We can suppose $\lambda < \lambda^{<\theta}$. Then by Theorem~\ref{SPVariants}(B), $SP^1_{T_{rg}}(\lambda, \lambda, \theta)$ holds. By Theorem~\ref{SCHTheorem}, there is $\theta \leq \mu < \lambda$ with $\mu = \mu^{<\theta}$ and $2^\mu \geq \lambda$. By (B), $SP^1_T(\lambda, \mu, \theta)$ holds, so by Theorem~\ref{SPVariants}(A), $SP_T(\lambda, \theta)$ holds, as desired.
\end{proof}
\newpage

\section{Conclusion}\label{Conclusion}
\bigskip

We begin to put everything together. We aim to produce a forcing extension in which, whenever $T$ has $<k$-type amalgamation, then $T_{k, k-1} \not \leq_{SP} T$. We will choose in advance nice pairs $(\theta_k, \lambda_k)$ to witness this. In order to arrange that $SP_T(\lambda_k, \theta_k)$ holds we will use Theorems~\ref{SatPortionPt2} and \ref{SatPortion}. To arrange that $SP_{T_{k, k-1}}(\lambda_k, \theta_k)$ fails, we will use the following.
\begin{theorem}
\label{NonSatPortion}
Suppose $(\theta, \lambda)$ is a nice pair such that $\theta = \theta^{<\theta}$ and $\lambda > \theta$ is a limit cardinal.  Let $3 \leq k < \omega$. Then $P_{\lambda \theta \theta}$ forces that for all $\mu < \lambda$, $\check{T}_{k+1, k}$ fails $(<k+1, \lambda, \mu, \theta)$-type amalgamation.
\end{theorem}

\begin{proof}
Fix $\theta \leq \mu < \lambda$, and write $P = P_{[\lambda]^k \theta \theta}$. We show that $P$ forces $\check{T}_{k+1, k}$ fails $(<k+1, \lambda, \mu, \theta)$-type amalgamation. Since $P \cong P_{\lambda \theta \theta}$, this suffices.

We  pass to a $P$-generic forcing extension $\mathbb{V}[G]$ of
$\mathbb{V}$. Let $R \subseteq [\lambda]^k$ be the set of all $v$ with
$\{(v, 0)\} \in G$. Choose $M_0 \preceq M \models T_{k+1, k}$, and
$(a_{i, \alpha}: i < k, \alpha < \lambda)$ such that, writing
${\overline{a}}_s = \{a_{i,\alpha}:(i,\alpha) \in s\}$ for 
$s \subseteq k \times \lambda$:
\mn
\begin{itemize}
\item $M_0$ is countable, and $|M| \leq \lambda$ and each $a_{i,
    \alpha} \in M \backslash M_0$;
\sn
\item $a_{i, \alpha} = a_{j, \beta}$ iff $\alpha = \beta$ and $i = j$
\sn
\item For every $v_* \in [k \times \lambda]^k$, if $v_*$ is not the graph of the increasing enumeration of some $v \in [\lambda]^k$, then $R^M({\overline{a}}_{v_*})$ fails. Otherwise, $R^M({\overline{a}}_{v_*})$ holds if and only if $v \in R$.
\end{itemize}

For each $v \in [\lambda]^{k}$, let $\phi_v(x, {\overline{a}}_{k
  \times v})$ be the formula that asserts that $R(x,
{\overline{a}}_{u})$ holds for each $u \in [k \times v]^{k-1}$. 
Note that 
$\phi_v(x, \overline{a}_{k \times v})$ is consistent exactly when $v \not \in R$.

It suffices to show that there is no cardinal $\lambda'$ and function $F_0: \Gamma^\theta_{M, M_0} \to^w_{k+1} P_{\lambda' \mu \theta}$; so suppose towards a contradiction some such $F_0$ existed. Then we can find $F: [\lambda]^{k}\backslash R \to  P_{\lambda' \mu \theta}$ such that for all sequences $(w_i: i < k+1)$ from $[\lambda]^k \backslash R$, if $(F(w_i): i < k)$ is compatible in $P$ then $\bigwedge_{i < k} \phi_{w_i}(x, \overline{a}_{k \times w_i})$ is consistent. This is all we will need, and so we can replace $\lambda'$ by $\lambda$ (since $|[\lambda]^k| = \lambda$). 

Pulling back to $\mathbb{V}$, we can find $p_* \in P$, and $P$-names $\dot{R}, \dot{M}, \dot{M}_0, \dot{a}_{i, \alpha}, \dot{F}$, such that $p_*$ forces these behave as above. 

Write $X = \lambda \backslash \bigcup \mbox{dom}(p_*)$; so $|X| =\lambda$.

Suppose $v \in [X]^{k}$. Choose $p_v \in P$ such that $p_v \geq p_* \cup \{(v, 1)\}$ (so $p_v$ forces $v \not \in \dot{R}$), and so that $p_v$ decides $\dot{F}(v)$, say $p_v$ forces that $\dot{F}(v) = f_v \in P_{\lambda \mu \theta}$.

Choose $F_*: [\lambda]^{k} \to P_{\lambda \mu \theta}$ so that for all $v, v'$, if $F_*(v)$ and $F_*(v')$ are compatible, then $p_v, p_{v'}$ are compatible, and $f_v, f_{v'}$ are compatible.

For each $u \in [\lambda]^{k-1}$ let $\mathcal{P}_u = \{F_*(v): v \in [\lambda]^k, u \subseteq v\}$ and let $\mathcal{Q}_u$ be the set of all $g \in P_{\lambda \mu \theta}$ such that $g$ extends some $f \in \mathcal{P}_u$. Choose a maximal antichain $(g_{u, \alpha}: \alpha < \kappa_u)$ from $\mathcal{Q}_u$. For each $\alpha < \kappa_u$ choose $w_{u, \alpha} \in \mathcal{P}_u$ such that $g_{u, \alpha}$ extends $F_*(w_{u,\alpha})$.

Since $P$ has the $\mu^+$-c.c., we have that each $\kappa_u \leq \mu$. For $u \in [\lambda]^{k-1}$ let $S(u) \in [\lambda]^\mu$ be sufficiently large so that each $w_{u, \alpha} \in [S(u)]^k$ and $\mbox{dom}(p_{w_{u, \alpha}}) \subseteq [S(u)]^k$.

By Theorem 46.1 of \cite{Combinatorics}, using $\lambda \geq \mu^{+\omega}$, we can find some $v \in [\lambda]^k$ such that for all $u \in [v]^{k-1}$, $S(u) \cap v = u$. Enumerate $[v]^{k-1} = (u_i: i < k)$. By induction on $i < k$ we pick $w_i \in \mathcal{P}_u$ such that $(F_*(v), F_*(w_j): j \leq i)$ is compatible in $P_{\lambda \mu \theta}$. To see this is possible, suppose we have $w_j: j < i$. Put $f = F_*(v) \cup \bigcup\{F_*(w_j): j < i\}$. Since $f$ extends $F_*(v)$, we have $f \in \mathcal{Q}_{u_i}$; by maximality of the antichain $(g_{u_i, \alpha}: \alpha < \kappa_{u_i})$, we must have that $f$ is compatible with some $g_{u_i, \alpha}$, and hence with $F_*(w_{u_i, \alpha})$, so put $w_i = w_{u_i, \alpha}$.

 Writing $v_{u_i} := w_i$, we have found $(v_u: u \in [v]^{k-1})$ such that each $u \subseteq v_u \in [S(u)]^k$, and $(F_*(v_u): u \in [w]^{k-1})$ is compatible. Thus $(p_{v_u}: u \in [v]^{k-1})$ is compatible in $P$; write $p = \bigcup_{u \in [v]^{k-1}} p_{v_u}$. Note that $v \not \in \mbox{dom}(p)$, since if $v \in \mbox{dom}(p_{v_u})$ then $v \in [S(u)]^k$, contradicting that $S(u) \cap v = u$. Thus we can choose $p' \geq p$ in $P$ with $p'(v) = 0$.

Now $p'$ forces that each $\dot{F}(v_u) = \check{f}_{v_u}$, and
$(f_{v_u}: u \in [v]^{k-1})$ is compatible;  thus $p'$ forces that
$\phi(x) := \bigwedge_{u \in [v]^{k-1}}
\phi_{v_u}(x, \dot{\overline{a}}_{k \times v_u})$ is consistent. But this
is impossible, since if we let $v_*$ be the graph of the increasing
enumeration of $v$, then $p'$ forces that
$\dot{R}^{\dot{M}}(\dot{\overline{a}}_{v_*})$ holds, and $\phi(x)$ in
particular implies that $\dot{R}^{\dot{M}}(x,
\dot{\overline{a}}_{u_*})$ holds 
for all $u_* \in [v_*]^{k-1}$, thus creating a $k$-clique.
\end{proof}

\begin{lemma}\label{TlkLemma}
Suppose $(\lambda, \theta)$ is a nice pair and $\ell > k \geq 3$. Then $\SP^1_{T_{\ell, k}}(\lambda, \lambda, \theta)$ if and only if $\SP_{T_{\ell,k}}(\lambda, \theta)$.
\end{lemma}
\begin{proof}
Forward direction is Theorem~\ref{SPVariants}(A). The reverse direction is like Theorem~\ref{SPVariants}(B): suppose $M \models T_{\ell,k}$ has $|M| \leq \lambda$ and $M_* \leq M$ is countable. Choose $N \succeq M$, a $\theta$-saturated model of size $\lambda$. Let $a_\alpha: \alpha < \lambda$ enumerate $N$. For each $\alpha < \lambda$ let $p_\alpha(x)$ be the type over $M$ asserting $x \not= a$ for each $a \in M$, and $R(x, \overline{a}) \in p_\alpha$ if and only if $R(a_\alpha, \overline{a})$ holds for each $\overline{a} \in [M]^{k-1}$. This is a complete type over $M$ which does not fork over $\emptyset$. Then $\{p_\alpha(x): \alpha < \lambda\}$ along with all realized types over $M_*$ witness $\SP^1_{T_{\ell,k}}(\lambda, \lambda, \theta)$.
\end{proof}

\begin{lemma}\label{InfManyClassesLemma}
Suppose $(\theta, \lambda)$ is a nice pair with $2^{\aleph_0} < \cof(\lambda) < \theta$. Suppose $\theta^{<\theta} = \theta$ and suppose there is $\kappa > \lambda^{<\theta}$ with $\kappa^{<\kappa} = \kappa$. Suppose $k \geq 3$. Then there is a forcing notion $P$ with $|P| = \kappa$ which is $\theta$-closed, $\theta^+$-c.c. and which forces: $(\theta, \lambda)$ is a nice pair, and for all $T$ with $<k$-type amalgamation, $\SP_T(\lambda, \theta)$ holds, and $\SP_{T_{k, k-1}}(\lambda, \theta)$ fails.
\end{lemma}
\begin{proof}
Note that any $P$ which is $\theta$-closed, $\theta^+$-c.c. will force that $(\theta, \lambda)$ is a nice pair.

Let $P = P_{\lambda \theta \theta}$. By Theorem~\ref{ForcingAxiomsCons} we can choose a $P$-name $\dot{Q}$ such that $P$ forces $\dot{Q}$ has the $(<k, \theta, \theta)$-amalgamation property, and forces Ax$(<k, \theta)$ to hold, and forces $2^{\theta} = \kappa$. We claim $P * \dot{Q}$ works.

If $T$ has $<k$-type amalgamation, then by Theorem~\ref{SatPortion} and Theorem~\ref{SatPortionPt2}, $P* \dot{Q}$ forces that $\SP^1_T(\lambda, \theta, \theta)$ holds, thus by Theorem~\ref{SPVariants}(A), $P* \dot{Q}$ forces $\SP_T(\lambda, \theta)$. 

So it suffices to show $P* \dot{Q}$ forces that $\SP_{T_{k, k-1}}(\lambda, \theta)$ fails. If $k = 3$ then this follows from $\lambda < \lambda^{<\theta}$ and the fact that $T_{3,2}$ is non-simple, by Theorem~\ref{Old}(A). So suppose $k \geq 4$. 

By Theorem~\ref{NonSatPortion} and Theorem~\ref{PresOfFail}, $P* \dot{Q}$ forces that $\SP^1_{T_{k,k-1}}(\lambda, \mu, \theta)$ fails for all $\mu < \lambda$.  By Theorem~\ref{SPVariants}(D), $P* \dot{Q}$ forces that $\SP^1_{T_{k, k-1}}(\lambda, \lambda, \theta)$ fails; by Lemma~\ref{TlkLemma}, $P * \dot{Q}$ forces that $\SP_{T_{k, k-1}}(\lambda, \theta)$ fails.

\end{proof}
\begin{theorem}\label{InfManyClasses}
Suppose $GCH$ holds. Then there is a forcing notion $P$, which forces: for every $k \geq 3$, if $T$ is a simple theory with $< k$-type amalgamation, then $T_{k, k-1} \not \leq_{SP} T$. In particular, the non-simple theories are exactly the maximal $\leq_{SP}$-theories.

\end{theorem}
Of course, we can also force to make GCH hold (via a proper-class forcing notion). Thus, this can consistently hold.
\begin{proof}
The ``in particular" clause follows since simple theories have $<3$-type amalgamation and non-simple theories are maximal by Theorem~\ref{Old}(A).

 Choose nice pairs $((\theta_n, \lambda_n): n < \omega)$, such that each $\theta_{n+1} > \lambda_{n}^{++}$ is regular, and each $\lambda_n$ is singular with $2^{\aleph_0} < \mbox{cof}(\lambda_n) < \theta_n$ (so each $\lambda_n^{<\theta_n} = \lambda_n^+$). 

We will define a full-support forcing iteration $(P_n: k \leq \omega)$, $(\dot{Q}_n: k < \omega)$; for each $n < \omega$, we will have that $|P_n| \leq \lambda_{n-1}^{++}$, and $P_n$ will force that $\dot{Q}_n$ is $\theta_n$-closed and has the $\theta_n^+$-c.c. Having defined $P_n$, we will inductively have that $|P_n| < \theta_n$ and is $(2^{\aleph_0})^+$-closed, and so $P_n$ forces $(\theta_n, \lambda_n)$ is a nice pair with $2^{\aleph_0} < \cof(\lambda) < \theta_n$ and $\theta_n^{<\theta_n} = \theta_n$ and GCH holds above $|P_n|$. Let $\dot{Q}_n$ be as supplied by Lemma~\ref{InfManyClassesLemma} with $\theta = \theta_n, \lambda = \lambda_n, \kappa = \lambda_n^{++}, k = n + 3$.

Then each $P_{n+1}$ forces that for all $T$ with $<n+3$-type amalgamation, $\SP_T(\lambda_n, \theta_n)$ holds, and $\SP_{T_{n+3, n+2}}(\lambda_n, \theta_n)$ fails. Since $\dot{Q}_{\geq n+1}$ is forced to be $\lambda_n^{++}$-closed, it does not disturb this, so we are done.

\end{proof}
\bigskip

\providecommand{\href}[2]{#2}

% \bibliographystyle{amsalpha}
% \bibliography{lista,listb,listc,listd,liste,listf,listv,listx,listy,listz}

\end{document}